\documentclass{amsart}
\usepackage{amsthm}
\usepackage{amsmath}
\usepackage{amsfonts}
\usepackage{amssymb}
\usepackage[usenames]{color}
\usepackage[mathscr]{euscript}
\usepackage[hidelinks]{hyperref} 
\def\MakeStep#1{\par\noindent\texttt{#1}}
\newcount\MyOutlineCount                
\def\MyOutline#1{\ifnum\MyOutlineCount>0 #1\fi}
\def\natural{{\mathbb N}}

\def\real{{\mathbb{R}}}
\def\sphere{{\mathbb{S}}}
\def\ball#1,#2.{B(#1,#2)} 
\def\clball#1,#2.{\bar B(#1,#2)} 
\def\uball#1,#2,#3.{B_{#3}(#1,#2)} 
\def\cluball#1,#2,#3.{\bar B_{#3}(#1,#2)} 
\def\lebmeas#1.{{\mathcal L}^{\setbox0=\hbox{$#1\unskip$}\ifdim\wd0=0pt 1
    \else #1\fi}} 
\def\hmeas#1.{\mathscr{H}^{\setbox0=\hbox{$#1\unskip$}\ifdim\wd0=0pt 1
    \else #1\fi}} 
\newcommand{\on}{\:\mbox{\rule{0.1ex}{1.2ex}\rule{1.1ex}{0.1ex}}\:}
\def\Xint#1{\mathchoice
   {\XXint\displaystyle\textstyle{#1}}%
   {\XXint\textstyle\scriptstyle{#1}}%
   {\XXint\scriptstyle\scriptscriptstyle{#1}}%
   {\XXint\scriptscriptstyle\scriptscriptstyle{#1}}%
   \!\int}
\def\XXint#1#2#3{{\setbox0=\hbox{$#1{#2#3}{\int}$}
     \vcenter{\hbox{$#2#3$}}\kern-.5\wd0}}

\def\av{\Xint-}
\def\glip#1.{{\bf L}(#1)} 
\def\glipdec#1,#2.{{\bf L}_{#2}(#1)}
\def\biglip{\text{\normalfont Lip}} 
\def\biglipdec#1.{\text{\normalfont Lip_{#1}}}
\def\smllip{\text{\normalfont lip}} 
\def\pcreatenrm#1#2{\expandafter\def\csname
  #1nrm\endcsname##1.{\left\|##1\right\|_{#2}} \expandafter\def\csname
  #1nrmname\endcsname{\left\|\,\cdot\,\right\|_{#2}}} 
\pcreatenrm{hilb}{l^2}
\DeclareMathOperator\dom{dom}
\DeclareMathOperator\diam{diam}
\DeclareMathOperator\len{len}
\DeclareMathOperator\slen{slen} 

\def\dym#1,#2.{\setbox1=\hbox{$#1\unskip$}\text{\normalfont{Dym}}_{\ifdim\wd1=0pt
    n \else #1\fi}(#2)}
\def\gates#1.{\setbox1=\hbox{$#1\unskip$}\text{\normalfont{Gates}}({\ifdim\wd1=0pt
    Q \else #1\fi})}
\def\cell#1.{\setbox1=\hbox{$#1\unskip$}\text{\normalfont{Cell}}({\ifdim\wd1=0pt
    Q \else #1\fi})}
\def\jumpp#1.{\setbox1=\hbox{$#1\unskip$}\text{\normalfont{Jpp}}({\ifdim\wd1=0pt
    Q \else #1\fi})}
\def\tjumpp#1.{\setbox1=\hbox{$#1\unskip$}\text{\normalfont{TJpp}}({\ifdim\wd1=0pt
    Q \else #1\fi})}
\def\chrom#1.{\setbox1=\hbox{$#1\unskip$}\text{\normalfont{Chr}}({\ifdim\wd1=0pt
    Q \else #1\fi})}
\def\lift#1.{\setbox1=\hbox{$#1\unskip$}\text{\normalfont{Lift}}({\ifdim\wd1=0pt
    \gamma_l \else #1\fi})}
\def\ccover#1.{\setbox1=\hbox{$#1\unskip$}\text{\normalfont{Cov}}({\ifdim\wd1=0pt
    Q \else #1\fi})}
\def\todoub#1.{\setbox1=\hbox{$#1\unskip$}\text{\normalfont{ToDouble}}({\ifdim\wd1=0pt
    Q \else #1\fi})}
\def\subdiv#1.{\setbox1=\hbox{$#1\unskip$}\text{\normalfont{Subdiv}}({\ifdim\wd1=0pt
    Q \else #1\fi})}
\def\chrsheet{\mathcal{C}}
\def\covsheet{\mathcal{B}}
\def\gtsheet{\mathcal{G}}
\def\disc#1,#2.{\setbox1=\hbox{$#1\unskip$}\setbox2=\hbox{$#2\unskip$}
  \text{\normalfont{Disc}}(\ifdim\wd1=0pt \varepsilon \else #1\fi, 
  \ifdim\wd2=0pt r \else #2\fi)}
\def\grid#1.{\text{\normalfont{Grid}}(#1)}
\def\fund#1,#2,#3.{\setbox2=\hbox{$#2\unskip$}\setbox3=\hbox{$#3\unskip$}
     \text{\normalfont Fund}(#1, \ifdim\wd2=0pt \varepsilon \else #2 \fi, 
   \ifdim\wd3=0pt r \else #3 \fi)}
\def\fundsub#1,#2,#3,#4.{\setbox2=\hbox{$#2\unskip$}\setbox3=\hbox{$#3\unskip$}
     \text{\normalfont Fund}_{#4}(#1, \ifdim\wd2=0pt \varepsilon \else #2 \fi, 
     \ifdim\wd3=0pt r \else #3 \fi)}
\def\skeleton#1,#2.{\setbox1=\hbox{$#1\unskip$}\setbox2=\hbox{$#2\unskip$}
  \text{\normalfont SK}_{\ifdim\wd1=0pt 2 \else #1\fi}(\ifdim\wd2=0pt
  X_j\else #2\fi)}
\def\hskeleton#1,#2.{\setbox1=\hbox{$#1\unskip$}\setbox2=\hbox{$#2\unskip$}
  \text{\normalfont HSK}_{\ifdim\wd1=0pt 2 \else #1\fi}(\ifdim\wd2=0pt
  X_j\else #2\fi)}
\def\hharmonic#1,#2.{\setbox1=\hbox{$#1\unskip$}\setbox2=\hbox{$#2\unskip$}
  \text{\normalfont H}_{\ifdim\wd1=0pt j \else #1\fi}(\ifdim\wd2=0pt f
  \else #2\fi)}
\def\gharmonic#1,#2.{\setbox1=\hbox{$#1\unskip$}\setbox2=\hbox{$#2\unskip$}
  \text{\normalfont G}_{\ifdim\wd1=0pt j \else #1\fi}(\ifdim\wd2=0pt f
  \else #2\fi)}
\def\ph{\text{\normalfont ph}}
\def\jp{\text{\normalfont jp}} 
\def\jph{\text{\normalfont jph}} 
\def\badset#1.{\setbox1=\hbox{$#1\unskip$}\text{\normalfont{Bad}}({\ifdim\wd1=0pt
    X_{l-1} \else #1\fi})}
\DeclareMathOperator\cpend{end}
\DeclareMathOperator\cpstart{sta}
\def\Eudec{_\text{\normalfont Euc}}
\def\greendec{_\text{\normalfont {green}}}
\def\egreendec{_{\text{\normalfont {green}},\varepsilon}}
\def\reddec{_\text{\normalfont {red}}}
\def\hardec{_\text{\normalfont {Har}}}
\def\baddec{_\text{\normalfont {bad}}}
\DeclareMathOperator\sgn{sgn}
\def\proj#1,#2.{\setbox1=\hbox{$#1\unskip$}\setbox2=\hbox{$#2\unskip$}\pi_{{\ifdim\wd1=0pt
      \infty\else #1\fi}, {\ifdim\wd2=0pt
      0\else #2\fi}}}
\def\projinv#1,#2.{\setbox1=\hbox{$#1\unskip$}\setbox2=\hbox{$#2\unskip$}\pi^{-1}_{{\ifdim\wd1=0pt
      \infty\else #1\fi}, {\ifdim\wd2=0pt
      0\else #2\fi}}}
\def\munsplit{8mu} 
\numberwithin{equation}{section} 
\theoremstyle{plain}
\newtheorem{lem}[equation]{Lemma}

\newtheorem{thm}[equation]{Theorem}

\theoremstyle{definition}
\newtheorem{defn}[equation]{Definition}
\newtheorem{constr}[equation]{Construction}
\newtheorem{macrotest}[equation]{Macro Test}
\theoremstyle{remark}

\newtheorem{rem}[equation]{Remark}
\begin{document}
\title[PI-unrectifiability]{An example of a differentiability space which is PI-unrectifiable}
\author{Andrea Schioppa}
\begin{abstract}
  We construct a (Lipschitz) differentiability space which has at
  generic points a disconnected tangent and thus does not contain
  positive measure subsets isometric to positive measure subsets of
  spaces admitting a Poincar\'e inequality. We also prove that
  $l^2$-valued Lipschitz maps are differentiable a.e., but there are
  also Lipschitz maps taking values in some other Banach spaces having
  the Radon-Nikodym property which fail to be differentiable on sets
  of positive measure.
\end{abstract}
\maketitle
\tableofcontents
\section{Introduction}
\label{sec:intro}
\subsection{Overview}
\label{subsec:overview}
\MyOutline{
\begin{enumerate}
\item Opening crap $\checkmark$
\item HK, Ch, ChKl $\checkmark$
\item Keith $\checkmark$
\item Surge of activity $\checkmark$
\end{enumerate}
}
This paper deals with the foundations of first-order calculus in
metric measure spaces. In this work we empirically address the
question of whether a Poincar\'e inequality is needed, at the
infinitesimal level, to have a Rademacher-like Theorem on the
a.e.~differentiability of Lipschitz functions. These results were
announced in~\cite{schioppa_neck_poinc}.
\par In the seminal~\cite{heinonen98} Heinonen and Koskela
introduced the notion of PI-spaces, i.e.~a class of metric measure
spaces which satisfy an abstract version of the Poincar\'e
inequality. In the remarkable~\cite{cheeger99} Cheeger proved that
in PI-spaces it is possible to develop first-order calculus;
specifically he proved a version of Rademacher's Theorem on the
a.e.~differentiability of real-valued Lipschitz
functions. Later~\cite{cheeger_kleiner_radon} Cheeger and Kleiner were even able, for PI-spaces,
to prove the a.e.~differentiability of Lipschitz functions which take
value in Banach spaces having the Radon-Nikodym property.
\par In~\cite{keith04} Keith introduced an analytic condition,
the \textbf{Lip-lip inequality} (later shown to self-improve to an equality)
and used it to prove a Rademacher Theorem on the
a.e.~differentiability of real-valued Lipschitz functions; the spaces
satisfying the conclusion of the differentiability theorem of Keith
will be called here \textbf{differentiability spaces} (the structure
was named by Keith \emph{(strong) measurable differentiable
  structure}, but had been actually singled out without giving it a name by
Cheeger in~\cite{cheeger99}).
\par Keith claimed to have generalized Cheeger's result. From the
technical standpoint his claim was factual: the Lip-lip inequality
appears to be a weaker condition than the Poincar\'e inequality and so
his argument his more general than the one given in~\cite{cheeger99}. However, his claim was not supported by
\emph{empirical evidence}: to the best of our knowledge all known examples
of differentiability  spaces are \textbf{PI-rectifiable}; i.e.~they
can be decomposed into a countable union of positive-measure subsets
of PI-spaces.
\par In the last five years there has been a surge of work on
differentiability spaces: \cite{bate_speight,bate-diff,deralb,dim_blow,cks_metric_diff,bate_rnp_new_definition,sylvester-rnp}. Despite this theoretical progress there is a
substantial gap between the theory and the structural properties exhibited by
known examples. This has led to the following
question~\cite{cks_metric_diff} (here rephrased as a conjecture); 
disclaimer: to the best of my knowledge~\cite{cks_metric_diff}
is the first place where the question has been
written down; I learnt it from Bruce Kleiner, but it is also possible
to have been considered before by others, for example
in~\cite{heinonen07} Heinonen said it was important to understand the
conditions needed to have a Rademacher Theorem.
\begin{description}
\item[(ConjPIRect)] Any differentiability space is PI-rectifiable; in
  particular a.e.~its tangents/blow-ups are PI-spaces.
\end{description}
The \textbf{(ConjPIRect)} has been recently proved in the
beautiful~\cite{sylvester-rnp}
 \emph{under the additional assumption} that
Lipschitz functions taking values in Banach spaces with the
Radon-Nikodym property are differentiable. In this paper our goal is
to \emph{disprove} \textbf{(ConjPIRect)}.
\subsection{The Result}
\label{subsec:result}
In this paper we construct an example of a metric measure space
$(X_\infty,\mu_\infty)$ such that:
\begin{description}
\item[(TritanopeExa)] Any Lipschitz map $f:X_\infty\to l^2$ is
  differentiable $\mu_\infty$-a.e., but $(X_\infty,\mu_\infty)$ is
  PI-unrectifiable; moreover at $\mu_\infty$-a.e.~point it has
  \emph{a} tangent which is \emph{not} topologically connected.
\end{description}
The fact that at $\mu_\infty$-a.e.~there is a topologically
disconnected tangent is Theorem~\ref{thm:disc-tangs}; this immediately
implies that $X_\infty$ cannot contain a positive measure subset $S$
of a PI-space: at $\mu_\infty$-a.e.~$p\in S$ the tangents of
$X_\infty$ would then be PI-spaces, which are known~\cite{cheeger99}
to be quasi-convex and hence connected. 
\par For expository reason we first prove the
$\mu_\infty$-a.e.~differentiability of real-valued Lipschitz
functions, Theorem~\ref{thm:real-diff}, and reserve the more technical
details for $l^2$-valued maps to Theorem~\ref{thm:l2-diff}.
\subsection{Outline}
\label{subsec:outline}
$X_\infty$ is constucted as an \textbf{inverse limit system}. The
basic operation is similar to Example~1.2 of~\cite{cheeger_inverse_l1}
(see~\cite{lang_plaut} for the metric properties of the space
and~\cite{cheeger_inverse_poinc}for the proof the Poincar\'e
inequality) or the
\emph{Laaksofolds} of~\cite{lee-laaksofolds}. Here we essentially
double a $3$-dimensional cell generating a diamond-like space,
Construction~\ref{constr:nqc-diamond}. However we do \emph{not} take
the path metric, but \emph{squeeze} closely the centers of the two
cells: this destroys the connectedness of some tangents.
\par The proof of differentiability requires new ideas as the usual
arguments~\cite[\# 3.1.6]{federer_gmt}
or\cite{cheeger_kleiner_radon,bate_rnp_new_definition} 
require joining
pairs of points by quasi-geodesics and constructing the derivative on
these curves. Our argument is \emph{functional}: we show that if
$f:X_\infty\to l^2$ is Lipschitz, it must \emph{eventually} collapse the centers of the
doubled cells faster than they are separated in the ambient space,
compare Theorems~\ref{thm:gate-collapse}
and~\ref{thm:l2-gate-collapse}. This collapsing argument is based on
some elementary PDE, see Lemmas~\ref{lem:energy-lwbd}
and~\ref{lem:l2-energy-lwbd} and might be regarded as a
\emph{tail-recursive} version of quantitative differentiation, see~\cite{cheeger-quant-diff}. 
\subsection{Questions}
\label{subsec:questions}
Here are some questions that hopefully can give the reader some food
for thought.
\begin{description}
\item[(Q1)] Our example has \textbf{analytic dimension} 3, i.e.~the
  gradient has three components. The techniques of the
  forthcoming~\cite{schioppa-pi-prescr}
 suggest that one can (with a lot of
  technical overhead) modify this example to get analytic dimension
  1. However, at the moment we only have examples where the
  Assouad-Nagata dimension~\cite{lang_nagata} is 3, can it be
  lowered to 1?
\item[(Q2)] What is the relationship between the
  $\mu_\infty$-a.e.~differentiability for $l^1$ and $l^2$-valued maps?
\end{description}
Note by~\cite{bate_rnp_new_definition} there are a Banach space $B$ having the
Radon-Nikodym property and a non-a.e.~differentiable Lipschitz map
$f:X_\infty\to B$. As remarked in~\cite{dim_blow} the
construction in~\cite{bate_rnp_new_definition} can be slightly improved to
yield non-differentiability in a canonical Banach space having the
Radon-Nikodym property:
\begin{equation}
  \label{eq:Sem}
  \text{Sem} = \bigoplus_{\substack{n=1\\ l^1}}^\infty l_n^\infty,
\end{equation}
the $l^1$-sum of copies of $\real^n$, with the $l^\infty$-norm, whose
dimension progressively increases to $\infty$. Thus differentiability
in $\text{Sem}$ is stronger than in $l^2$, and \textbf{(Q2)} asks how
$l^1$ stands compared to $l^2$.
\subsection*{Notational conventions}
\label{subsec:notation}
We use the convention $a\simeq b$ to say that $a/b,b/a\in[C^{-1},C]$
where $C$ is a universal constant; we similarly use notations like
$a\lesssim b$ or
$a\gtrsim b$. The notation $\hmeas k.$ stands for the $k$-dimensional
Hausdorff measure and $\lebmeas k.$ for the $k$-dimensional Lebesgue
measure. Given a map $f:X\to Y$ and a measure $\mu$ on $X$,
$f_{\#}\mu$ denotes the push-forward of $\mu$ to a measure on $Y$;
finally $\av_Ag\,d\mu$ denotes the average of $g$ on $A$: $\int_Ag\,d\mu/\mu(A)$.
\subsection*{Acknowledgements}
\label{subsec:ackwn}
I wish to thank Urs Lang and Kevin Payne for discussions on quantitative
differentiation. This is work is based on numerical experiments on
computing asymptotical distributions of sums of Jones' $\beta$-numbers for
randomly sampled Lipschitz functions carried out in Python and R. I
acknowledge use of the packages \textbf{ggplot2}
(\cite{wickham-ggplot2}) and \textbf{seaborn} (built on top of
\cite{hunter-matplotlib}, \url{https://github.com/mwaskom/seaborn}). The author was supported by the ``ETH Zurich Postdoctoral Fellowship Program and the Marie Curie Actions
  for People COFUND Program''
\section{Background Material}
\label{sec:background}
\subsection{Functions and Spaces}
\label{subsec:funs-spaces}
\MyOutline{
\begin{enumerate}
\item The Lipschitz constants $\checkmark$
\item $l^2$-values harmonic functions $\checkmark$
\item Inverse limit systems $\checkmark$
\item Tangents and blow-ups $\checkmark$
\end{enumerate}
}
\begin{defn}[Lipschitz map]
  \label{defn:lip-map}
  Given a Lipschitz map $f:X\to Z$ between the metric spaces $X$ and
  $Z$, we let $\glip f.$ denote its global Lipschitz constant. To
  extract local information at $x$ on the Lipschitz constant we use
  the \textbf{big and small} Lipschitz constants:
  \begin{equation}
    \label{eq:lip-map-1}
    \begin{aligned}
      \biglip f(x) &= \limsup_{r\searrow0}\frac{1}{r}\sup_{y\in\ball
        x,r.}d_Z(f(x),f(y))\\
      \smllip f(x) &= \liminf_{r\searrow0}\frac{1}{r}\sup_{y\in\ball
        x,r.}d_Z(f(x),f(y)).
    \end{aligned}
  \end{equation}
\end{defn}
\begin{defn}[$l^2$-valued harmonic functions]
  \label{defn:harmonic-funs}
  Given $u:\Omega\subset\real^n\to l^2$ we say that it is
  \textbf{harmonic} if each component $u_j$ is harmonic. Given enough
  regularity on $\partial \Omega$ (e.g. $\partial\Omega$ is locally a
  Lipschitz graph), given a boundary condition $f:\partial\Omega\to
  l^2$, we can find a harmonic extension $u:\Omega\to l^2$: take a
  harmonic extension $u_j$ of each component $f_j$ of $f$, then
  observe that for each $n$:
  \begin{equation}
    \label{eq:harmonic-funs-1}
    \sum_{j=1}^nu_j^2
  \end{equation}
  is subharmonic and apply the maximum principle to conclude that if
  we set $u=(u_j)_{j=1}^\infty$ then $u$ is $l^2$-valued.
\end{defn}
\begin{defn}[Inverse Limit Systems]
  \label{defn:inv-sys}
  Let $(X_n)_{n=0}^\infty$ be a sequence of compact metric spaces with
  a uniform bound on their diameters:
  \begin{equation}
    \label{eq:inv-sys-1}
    \sup_n\diam X_n < \infty,
  \end{equation}
  and assume that there are surjective $1$-Lipschitz maps
  $\pi_{n+1,n}:X_{n+1}\to X_n$; then the inverse limit $X_\infty$ of
  $(X_n)_{n=0}^\infty$ consists of all the sequences
  $(x_n)_{n=0}^\infty$ satisfying $\pi_{n+1,n}(x_{n+1})=x_n$, and
  where the metric is defined by:
  \begin{equation}
    \label{eq:inv-sys-2}
    d_{X_\infty}(x_\infty, y_\infty) = \limsup_{n\nearrow\infty}d_{X_n}(x_n,y_n).
  \end{equation}
  Moreover, in this case we obtain a $1$-Lipschitz
  $\pi_{\infty,n}:X_\infty\to X_n$ just letting
  $(x_n)_{n=0}^\infty\mapsto x_n$.
  \par It is useful to get conditions under which $X_\infty$ is also
  the Gromov-Hausdorff limit of the $X_n$,
  compare~\cite{cheeger_inverse_l1,cheeger_inverse_poinc}. 
In this work it suffices to consider the
  following condition which just says that the $\pi_{\infty,n}$ give
  the desired Gromov-Hausdorff approximations:
  \begin{equation}
    \label{eq:GH-cond-inv-sys}
    \lim_{n\to\infty}\sup_{x\in X_n}\diam \projinv \infty,n.(x) = 0.
  \end{equation}
  Assume now that on each $X_n$ we have a Radon measure $\mu_n$ and
  that $\proj n+1,n\#.\mu_{n+1}=\mu_n$ and
  that~(\ref{eq:GH-cond-inv-sys}) holds. Then a standard compactness
  argument yields a Radon measure $\mu_\infty$ on $X_\infty$ such that
  $\proj \infty,n\#.\mu_\infty=\mu_n$ and $(X_n,\mu_n)$ converges to
  $(X_\infty,\mu_\infty)$ in the measured Gromov-Hausdorff sense. For
  a more general treatment of inverse limit systems of measure spaces
  we refer to~\cite{choksi-inverse-mealims}.
\end{defn}
\begin{defn}[Tangents/blow-ups]
  \label{defn:tangents}
  Let $X$ be a metric space and $p\in X$. \textbf{A tangent/blow-up}
  of $X$ at $p$ is a pointed metric space $(Y,q)$ which is the pointed
  Gromov-Hausdorff limit (you can see~\cite{dim_blow,burago-metbook}
  for a review of the basic properties of Gromov-Hausdorff
  convergence) of a sequence $(\frac{1}{r_n}X,p)$ where $r_n\searrow0$
  and $\frac{1}{r_n}X$ denotes the metric space $X$ with the rescaled
  metric:
  \begin{equation}
    \label{eq:tangents-1}
    d_{\frac{1}{r_n}X}(x,y)=\frac{1}{r_n}d_X(x,y).
  \end{equation}
  \par If $\mu$ is a Radon measure on $X$, a \textbf{measured}
  tangent/blow-up at $p$ is a pointed metric measure space $(Y,q,\nu)$
  such that $(\frac{1}{r_n}X,p,\mu/\mu(\uball p,r_n,X.))$ converges to
  $(Y,q,\nu)$ in the measured Gromov-Hausdorff sense.
\end{defn}
\subsection{Differentiability Spaces}
\label{subsec:diff-spaces}
\MyOutline{
\begin{enumerate}
\item Cheeger + version for Lip and gradient $\checkmark$
\item Conjecture $\checkmark$
\item RNP-differentiability space $\checkmark$
\item Summary of previous results: Cheeger, Keith, CK, BS, Bate,
  Schioppa, Bate-Li, EB $\checkmark$
\end{enumerate}
}
We start with a brief review of differentiability spaces.
For more details we refer to the original
papers~\cite{cheeger99,keith04} or to the nice expository
paper~\cite{kleiner_mackay}. 
This structure has several names in
the literature: \emph{(strong) measurable differentiabile structure,
differentiable structure (in the sense of Cheeger and Keith),
Lipschitz differentiability space, differentiability space}.
We highlight the features of differentiability spaces; contrary to
some earlier papers, we do not assume a uniform bound on the dimension of the
charts.
\begin{defn}
  \label{defn:diff-space}
Let $(X,\mu)$ be a metric measure space; we say that $X$ is a
\textbf{differentiability space} if:
\begin{description}
\item[(DiffChart)] There is a countable collection of charts
  $\{(U_\alpha,\phi_\alpha)\}_\alpha$, 
  where $U_\alpha\subset X$ is Borel and $\phi_\alpha:X\to\real^{N_\alpha}$ is Lipschitz,
  such that $X\setminus
  (\cup_\alpha U_\alpha)$ is $\mu$-null, and each real-valued
  Lipschitz function $f$ admits a first order Taylor expansion with
  respect to the components of $\phi_\alpha$
  at
  generic points of $U_\alpha$, i.e.~there are (a.e.~unique) measurable functions $\frac{\partial
      f}{\partial\phi_\alpha^i}$ on $U_\alpha$ such that:
  \begin{multline}
    \label{eq:taylor-exp-diff-space}
    f(x)=f(x_0)+\sum_{i=1}^{N_\alpha}\frac{\partial
      f}{\partial\phi_\alpha^i}(x_0)\left(\phi_\alpha^i(x)-\phi_\alpha^i(x_0)\right)
    +o\left(d(x,x_0)\right)\\ \quad(\text{for $\mu$-a.e.~$x_0\in U_\alpha$}).
  \end{multline}
  Equivalently:
  \begin{equation}
    \label{eq:lip-diff-diff-space}
    \biglip\biggl(
    f - \bigl\langle\frac{\partial
      f}{\partial\phi_\alpha^i}(x_0),\phi_\alpha^i\bigr\rangle
    \biggr)(x_0)=0.
  \end{equation}
\end{description}
The integer $N_\alpha$ is the \textbf{dimension} of the chart
  $\{(U_\alpha,\phi_\alpha)\}_\alpha$, and depends only on the set
  $U_\alpha$, not on the particular choice of the coordinate functions
  $\phi_\alpha$. If $\sup_\alpha N_\alpha<\infty$, it is called \textbf{the
  differentiability or the analytic dimension}. Note that $\{\frac{\partial
      f}{\partial\phi_\alpha^i}\}_{i=1}^{N_\alpha}$ are the components
    of the gradient $\nabla f$ with respect to the coordinate system
    $\{\phi_\alpha^i\}_{i=1}^{N_\alpha}$. 
\end{defn}
By~\cite{cheeger99}  to each differentiability space there are
associated measurable cotangent and tangent bundles
  $T^*X$ and $TX$; having locally trivialized $T^*X$ and $TX$, forms in $T^*X$
  correspond to differentials of Lipschitz functions, and vectors in
  $TX$ give rise to differential operators called derivations~\cite{weaver00,deralb}.
\par We now restate \textbf{(ConjPIRect)} in a more formal way.
\begin{description}
\item[(ConjPIRect)] Let $(X,\mu)$ be a differentiability space. Then
  there is a countable decomposition $X=\bigcup_iS_i\cup \Omega$ where
  $\mu(\Omega)=0$ and there are isometric embeddings $f_i:S_i\to Y_i$
  and measures $\nu_i$ on the spaces $Y_i$ such that $f_{i\#}\mu\on
  S_i=\nu_i\on f_i(S_i)$ and such that each $(Y_i,\nu_i)$ is a PI-space.
\end{description}
Following Bate and Li~\cite{bate_rnp_new_definition} we define
RNP-differentiability.
\begin{defn}[RNP-differentiability]
  \label{defn:rnp-diff}
  An \textbf{RNP-differentiability space} is a differentiability space
  where~(\ref{eq:taylor-exp-diff-space})
  and~(\ref{eq:lip-diff-diff-space}) hold also for any Lipschitz
  $f:X\to B$ where $B$ is a Banach space having the Radon-Nikodym property.
\end{defn}
\begin{thm}[Summary of results on differentiability spaces]
  \label{thm:diff-master-summary}
  This list summarizes relevant results on differentiability spaces:
  \begin{description}
  \item[(Cheeger)] \cite{cheeger99}; if $(X,\mu)$ is a PI-space
    then $(X,\mu)$ is a differentiability space whose analytic
    dimension is bounded by an expression that depends only on the
    doubling constant $C_\mu$ of $\mu$ and the constants that appear
    in the Poincar\'e inequality. Moreover, for each real-valued
    Lipschitz function $f$ one has $\biglip f = \smllip f$ $\mu$-a.e.
  \item[(Keith)] \cite{keith04}; assume that $(X,\mu)$ is a
    doubling metric measure space which
    satisfies the \textbf{Lip-lip inequality}: there
    is a constant $C\ge 1$ such that for each real-valued Lipschitz
    function $f$ one has $\biglip f\le C\smllip f$ $\mu$-a.e. Then
    $(X,\mu)$ is a differentiability space whose analytic dimension is
    bounded by an expression that depends only on $C_\mu$ and
    $C$. 
  \item[(Cheeger-Kleiner)] \cite{cheeger_kleiner_radon}; any PI-space is an
    RNP-differentiability space.
  \item[(Bate--Speight)] \cite{bate_speight}; if $(X,\mu)$ is a
    differentiability space then $\mu$ is asymptotically doubling in
    the sense that for $\mu$-a.e.~$x$ there are
    $(C_x,r_x)\in(0,\infty)^2$ such that:
    \begin{equation}
      \label{eq:diff_summary_s1}
        \mu\left(
          B(x,2r)
        \right) \le C_x
        \mu\left(
          B(x,r)
        \right)
        \quad(r\le r_x).
      \end{equation}
  \item[(Bate)] \cite{bate-diff}; any
     differentiability space $(X,\mu)$ can be decomposed, up to
     throwing away a null-set, into a
     countable union of positive measure subsets $S_i$ such that for
     each $(S_i,\mu\on S_i)$ the Lip-lip inequality holds with a
     constant $C_i=C(S_i)$.
  \item[(Schioppa)] \cite{deralb,dim_blow} A metric measure space $(X,\mu)$ is a differentiability space if and only
    if it satisfies the \textbf{Lip-lip equality}: given any
    real-valued Lipschitz function $f$ one has $\biglip f = \smllip f$
    $\mu$-a.e. Moreover, at $\mu$-a.e.~$p\in X$ all the measured
    tangents at $p$ are differentiability spaces.
  \item[(Cheeger-Kleiner-Schioppa)] \cite{cks_metric_diff}; for differentiability spaces a
    version of metric differentiation
    \cite{ambrosio_metric_bv,kirchheim_metric_diff,ambrosio-rectifiability}
 holds.
  \item[(Bate-Li)] \cite{bate_rnp_new_definition}; if $(X,\mu)$ is an
    RNP-differentiability space at $\mu$-a.e.~$p\in X$ the measured
    tangents of $X$ satisfy a \emph{non-homogeneous} Poincar\'e
    inequality.
  \item[(Ericksson-Bique)] \cite{sylvester-rnp};
    RNP-differentiability spaces are PI-rectifiable.
  \end{description}
\end{thm}
\section{The inverse limit system}
\label{sec:inv-lim}
In this section we focus on the metric measure properties of the
example. We first introduce the building blocks of the construction,
the non-quasiconvex diamonds, Construction~\ref{constr:nqc-diamond},
and then describe the inverse limit system,
Construction~\ref{constr:inv-lim-sys}. We then describe the shape of
balls, prove that the measures are doubling, Lemma~\ref{lem:ver-doubling}, and
construct the horizontal gradient, Definition~\ref{defn:hz-grad}. We also
introduce a way to discretize balls, the fundamental configuration of
Definition~\ref{defn:fund-epsi-r-conf} which is used in the proof of
differentiability. Then we construct horizontal paths with jumps,
Lemma~\ref{lem:good-hz-paths}, that connect a point at the center of a
ball with points in the fundamental configuration: these paths are a
key construction for proving differentiability. The proof of
differentiability will consist in controlling the variation of a Lipschitz
function in terms of two pieces: (1) the horizontal gradient on the
paths, and (2) a ``collapsing factor'' on the jumps. Finally, we prove the existence of
tangents which are not topologically connected, Theorem~\ref{thm:disc-tangs}.
\subsection{Building blocks}
\label{subsec:building-blocks}
\MyOutline{
\begin{enumerate}
\item Construction of non-quasiconvex diamond $\checkmark$
\item Construction of the inverse limit system $\checkmark$
\end{enumerate}
\begin{macrotest}
  $\dym ,[0,1]^3.$; $\dym 25, Q.$; $d\Eudec$, $\cell A.$, $\jumpp{\dym
    ,Q.}.$; $G\reddec$, $G\greendec$; $\gates B.$, $\proj ,.$, $\proj
  1,.$, $\proj ,2.$, $\proj 3,4.$.
\end{macrotest}
}
\begin{constr}[Non-quasiconvex diamonds]
  \label{constr:nqc-diamond}
  Let $[0,1]^3$ be the standard cube with the Euclidean metric
  $d\Eudec$ and $K=[1/2-1/6,1/2+1/6]^3$; fix an integer $n\ge 26$ and
  replace the subcube $K$ with two isometric copies $K_1, K_2$ glued
  along the boundary $\partial K$. Let $A = [0,1]^3\setminus K$ and
  note that both the spaces $A\cup K_1$ and $A\cup K_2$ are isometric
  to $[0,1]^3$. 
  \par As set, define $\dym ,[0,1]^3. = A\cup K_1\cup K_2$. Let $c_i$
  denote the center of $K_i$: as metric $d_{\dym ,[0,1]^3.}$ we take
  the largest metric which agrees with $d\Eudec$ on each $A\cup K_1$,
  $A\cup K_2$ and such that:
  \begin{equation}
    \label{eq:nqc-diamond-1}
    d_{\dym ,[0,1]^3.}(c_1,c_2)=\frac{1}{4n}.
  \end{equation}
  \par A more concrete description of $d_{\dym ,[0,1]^3.}$ can be
  obtained as follows. First define a symmetric function $\varrho:\dym
  ,[0,1]^3.\times \dym ,[0,1]^3.\to[0,\infty]$:
  \begin{equation}
    \label{eq:nqc-diamond-2}
    \varrho(p,q) =
    \begin{cases}
      d\Eudec(p,q) &\text{if $\{p,q\}\subset A\cup K_1$ or
        $\{p,q\}\subset A\cup K_2$}\\
      \frac{1}{4n} &\text{if $\{p,q\}=\{c_1,c_2\}$}\\
      +\infty &\text{otherwise.}
    \end{cases}
  \end{equation}
  Given $p,q\in\dym ,[0,1]^3.$ a \textbf{chain joining} $p$ to $q$ is a finite
  tuple $(p_0,\cdots p_N)$ with $p_0=p$, $p_N=q$; then:
  \begin{equation}
    \label{eq:nqc-diamond-3}
    d_{\dym ,[0,1]^3.}(p,q) = 
    \inf\left\{
      \sum_{i=0}^{N-1}\varrho(p_i, p_{i+1}): \text{$(p_0,\cdots p_N)$
        joins $p$ to $q$}
     \right\}.
   \end{equation}
   From~(\ref{eq:nqc-diamond-3}) it follows that the map $\pi:\dym
   ,[0,1]^3.\to[0,1]^3$ which collapses the two copies $K_1$, $K_2$
   together is $1$-Lipschitz.
   \par We now induce a cube-complex structure on $\dym ,[0,1]^3.$;
   choose $N=N(n)\in\natural$ such that if one subdivides $[0,1]$ into
   $2N+1$ intervals of the same length, this length lies in
   $[1/(128n),1/(32n)]$, and such that $2N+1$ is divisible by $3$. 
   Taking products of these intervals we obtain
   a cube-complex structure on $[0,1]^3$ where they all have the same
   side length. Moreover, one such cube, call it $\tilde K_N$, is
   centered at $(1/2,1/2,1/2)$. As $2N+1$ is divisible by $3$, this
   cube-complex structure is compatible with the doubling operation we
   applied to $K$, i.e.~ the boundary of $K$ will consist of $2$-cells
   and we can induce a cube-complex structure on
   $\dym,[0,1]^3.$ by adding the requirement that $\pi$ is open and
   cellular.
   \par Note that $\pi^{-1}(\tilde K_N) = \tilde K_{N,1} \cup\tilde
   K_{N, 2}$, the two cubes being centered at $c_1$, $c_2$
   respectively. We will call these cubes the \textbf{gates}, and let
   $\gates[0,1]^3. = \{ \tilde K_{N, 1}, \tilde K_{N, 2}\}$.
   \par Now to each copy of $K$ we attach a color, say $K_1$ is green
   and $K_2$ is red. We will then have a green gate
   $\gtsheet\greendec([0,1]^3)$ and a red gate
   $\gtsheet\reddec([0,1]^3)$. Similarly, we have two copies
   $\{\chrsheet_1,\chrsheet_2\}$ of
   $[0,1]^3$ inside $\dym ,[0,1]^3.$ depending on wether we choose
   green or red for the cover of $K$: we will call these the
   \textbf{chromatic sheets}, and let
   $\chrom[0,1]^3.=\{\chrsheet_1,\chrsheet_2\}$. Finally let the
   \textbf{cover sheets} $\ccover [0,1]^3.=\{\covsheet_1$,
   $\covsheet_2\}$ 
   be the same as $\{\chrsheet_1,\chrsheet_2\}$ (chromatic and
   cover sheets will differ at the next iterations of the
   construction).
   \par Finally, we let $\jumpp [0,1]^3.=\{(c_1, c_2)\}$ which we call
   the \textbf{jump pair} of $\dym ,[0,1]^3.$. 
   \par The construction described so far will be called the
   \textbf{nqc-diamond} on $[0,1]^3$ with parameter $n$. We can extend
   this construction to each cube $T$, obtaining $\dym ,T.$: just take
   a similarity and a translation that identify $T$ with $[0,1]^3$,
   perform the above construction, and then scale back the metric so
   that $\diam(\dym ,T.)=\diam(T)$.
   \par Given a measure $\mu_T$ on $T$ which is a sum of multiples of the
   Lebesgue measure on the 3-dimensional cells $\cell T.$ of $T$,
   there is a naturally induced measure $\mu_{\dym ,T.}$ 
on $\dym ,T.$
   such that $\pi_{\#}\mu_{\dym ,T.}=\mu_T$, which is obtained by
   splitting in $1/2$ the measure across pairs of cells of $\cell{ \dym
   ,T.}.$ that $\pi$ maps to the same cell of $T$.
\end{constr}
\begin{constr}[Construction of the inverse limit system]
  \label{constr:inv-lim-sys}
  Let $n_0\ge 100$ to be determined later, set $X_0=[0,1]^3$ with the
  standard Euclidean distance and Lebesgue measure $\mu_{X_0}=\lebmeas
  3.\on[0,1]^3$. For each integer $k$ let $n_k=n_0+k$ and let $\bar
  n_0=0$, $\bar n_1=n_1^3$, $\bar n_k=\sum_{j\le k}n_j^3$.
  \par\noindent\texttt{Step 1: The construction of $X_1, X_2, \cdots,
    X_{n_1^3}$.}
  \par We simply let $X_1=\dym n_1,X_0.$ and let $\proj 1,0.:X_1 \to X_0$ be
  the map $\pi$ as in Construction~\ref{constr:nqc-diamond} and
  $\mu_{X_1}$ the corresponding measure. Let $\slen(X_1)$ denote the
  common length of the sides of the elements of $\cell X_1.$. Set
  $\todoub X_0. = [0,1]^3$.
  \par To obtain $X_2$, let $\todoub X_1.= \cell X_1.\setminus\gates
  X_1.$ and for each $Q\in\todoub X_1.$,
  replace it with $\dym n_1,Q.$; on the other hand, subdivide each
  $Q\in\gates X_1.$ into smaller subcubes so that all the cells of
  $\cell X_2.$ have the same side-length $\slen(X_2)$; the set of the
  cells of $X_1$ which were only subdivided will be denoted by
  $\subdiv X_1.$.
  Combining the
  maps $\pi_Q:\dym n_1,Q.\to Q$ for $Q\in\todoub X_1.$ and the
  identity for $Q\in\subdiv X_1.$, we get a map $\proj
  2,1.:X_2\to X_1$; as in Construction~\ref{constr:nqc-diamond} we
  also obtain a measure $\mu_{X_2}$ with $\proj
  2,1\#.\mu_{X_2}=\mu_{X_1}$. Note that $\mu_{X_2}$ is a multiple
  of Lebesgue measure on each element of $\cell X_2.$.
  \par We now turn to a description of the metric $d_{X_2}$ of $X_2$
  by introducing the \textbf{chromatic sheets} $\chrom X_2.$ and the
  \textbf{cover sheets} $\ccover X_2.$. Take $X_1$: we can lift it
  to $X_2$ by choosing for each $Q\in\todoub X_1.$
  either the green or the red lift in the construction of $\dym
  n_1,Q.$; the set of all possible lifts of $X_1$ is $\ccover X_2.$
  and we have:
  \begin{equation}
    \label{eq:constr-inv-lim-sys-1}
    \#\ccover X_2. = 2^{\#\{Q\in\todoub X_1.\}}.
  \end{equation}
  For the moment note that we want $d_{X_2}$ so that for each
  $\covsheet\in\ccover X_2.$ one has that $\proj 2,1.:\covsheet\to X_1$ is
  an isometry. We now want to lift the chromatic sheets $\chrom X_1.$:
  this time for any lift $\tilde\chrsheet$ of $\chrsheet$ we are
  always choosing, across all $Q\in\todoub X_1.$ the
  \emph{same} color, either green or red. In particular, the set of
  all chromatic sheets $\chrom X_2.$ consists of $4$ elements, which
  we can label (green, green), (red, green), (green, red) and
  (red, red). Also, the set of lifts of $\chrsheet$ in $X_2$, $\chrom
  X_2,{\proj 2,1.^{-1}(\chrsheet)}.$ has cardinality $2$. For the
  moment note that we want $d_{X_2}$ so that for each $\chrsheet
  \in\chrom X_2.$ $\proj 2,0.:\chrsheet\to [0,1]^3$ is an
  isometry. Then we can restrict Lebesgue measure on each $\chrsheet$
  and get the representation:
  \begin{equation}
    \label{eq:constr-inv-lim-sys-2}
    \mu_{X_2} = \frac{1}{4}\sum_{\chrsheet\in\chrom X_2.}\lebmeas 3.\on\chrsheet.
  \end{equation}
  The set of \textbf{jump pairs} of $X_2$ is:
  \begin{equation}
    \label{eq:constr-inv-lim-sys-3}
    \jumpp X_2. = \bigcup_{Q\in\todoub X_1.}
      \jumpp Q.;
    \end{equation}
    moreover, note that for $(p,q)\in\jumpp X_1.$ we did not
    double the cells containing $p$ and $q$ and so we can regard
    $\{p,q\}$ as a subset of $X_2$ too. Define a symmetric
    function $\varrho:X_2\times X_2\to[0,\infty]$ by:
    \begin{equation}
      \label{eq:constr-inv-lim-sys-4}
      \varrho(p,q)=
      \begin{cases}
        d_{\covsheet\simeq X_1}(p,q)&\text{if
          $p,q\in\covsheet\in\ccover X_2.$}\\
        \frac{1}{4n_1}\slen(Q)&\text{if $\{(p,q)\}$ or $\{(q,p)\}=\jumpp Q.$ for
          $Q\in\todoub X_1.$}\\
    +\infty&\text{otherwise.}
      \end{cases}
    \end{equation}
    Then $d_{X_2}(p,q)$ is obtained by minimizing the cost of chains
    joining $p$ to $q$:
    \begin{equation}
      \label{eq:constr-inv-lim-sys-5}
      d_{X_2}(p,q) = 
    \inf\left\{
      \sum_{i=0}^{N-1}\varrho(p_i, p_{i+1}): \text{$(p_0,\cdots p_N)$
        joins $p$ to $q$}
     \right\}.
   \end{equation}
   Then $\proj 2,1.$ becomes $1$-Lipschitz, open and cellular, and
   satisfies the desiderata above. Finally note that $\gates X_1.$ can
   be identified with a subset of $X_2$ (as we did not apply the
   diamond construction on those cells of $X_1$ but just subdivided them)
   and we let:
   \begin{equation}
     \label{eq:constr-inv-lim-sys-6}
     \gates X_2. = \bigcup_{Q\in\todoub X_1.}\gates
    Q. \cup\bigcup_{\substack{Q\in\cell X_2.\\Q\subset\gates X_1.}}Q.
  \end{equation}
  \par Let $2\le k <\bar n_1$; to obtain $X_{k+1}$, we define $\todoub
  X_{k}.=\cell X_{k}.\setminus\gates X_k.$ and replace each
  $Q\in\todoub X_k.$ whith $\dym n_1,Q.$. The other definitions,
  e.g.~$d_{X_{k+1}}$, $\proj k+1,k.$ are as above.
  \par\noindent\texttt{Step 2: The construction of $X_{\bar n_k +1},
    X_{\bar n_k + 2}, \cdots,
    X_{\bar n_{k+1}}$.}
  \par To obtain $X_{\bar n_k +1}$ from $X_{\bar n_k}$, we apply the
  diamond construction to \emph{all} the cells of $X_{\bar n_k}$: we
  let $\todoub X_{\bar n_k}. = \cell X_{\bar n_k}.$ and replace each
  $Q\in\todoub X_{\bar n_k}.$ with $\dym n_{k+1},Q.$. We can lift $X_{\bar
    n_k}$ by choosing for each
  $Q\in\todoub
  X_{\bar n_k}.$ either the
  green or red lift in the construction of $\dym n_k,Q.$. We set of
  all possible lifts of $X_{\bar n_k}$ will be denoted by $\ccover
  X_{\bar n_k+1}.$ and we have:
  \begin{equation}
    \label{eq:constr-inv-lim-sys-7}
    \#\ccover X_{\bar n_k+1}. = 2^{\#\{Q\in\todoub X_{\bar n_k}.\}}.
  \end{equation}
  On the other hand, consider a chromatic sheet $\chrsheet\in\chrom
  X_{\bar n_k}.$; this admits exactly two lifts $\chrom X_{\bar
    n_k+1},{\projinv \bar n_k+1,\bar
    n_k.}(\chrsheet).=\{\chrsheet\greendec,
  \chrsheet\reddec\}$ where,
  whenever we choose a lift $Q\in\todoub X_{\bar n_k}.$, we \emph{always}
  choose either red or green. We also define the set of all chromatic
  sheets:
  \begin{equation}
    \label{eq:constr-inv-lim-sys-8}
    \chrom X_{\bar n_k+1}. = \bigcup_{\chrsheet\in\chrom X_{\bar n_k}.}\chrom X_{\bar
    n_k+1},{\projinv \bar n_k+1,\bar
    n_k.}(\chrsheet)..
\end{equation}
\par To construct $d_{X_{\bar n_k+1}}$ we proceed as before; we define a
symmetric function $\varrho: X_{\bar n_k+1}\times X_{\bar
  n_k+1}\to[0,\infty]$ by:
  \begin{equation}
    \label{eq:constr-inv-lim-sys-9}
    \varrho(p,q) =
    \begin{cases}
      d_{\covsheet\simeq X_{\bar n_k}}(p,q) &\text{if
        $p,q\in\covsheet\in\ccover X_{\bar n_k+1}.$}\\
      \frac{1}{4n_{k+1}}\slen(Q) &\text{if
        $\{(p,q)\}$ or $\{(q,p)\}=\jumpp Q.$ for $Q\in\todoub X_{\bar n_k}.$}\\
      +\infty&\text{otherwise.}
    \end{cases}
  \end{equation}
  Then  $d_{X_{\bar n_k+1}}(p,q)$ is obtained by minimizing the cost of chains
    joining $p$ to $q$:
    \begin{equation}
      \label{eq:constr-inv-lim-sys-10}
      d_{{X_{\bar n_k+1}}}(p,q) = 
    \inf\left\{
      \sum_{i=0}^{N-1}\varrho(p_i, p_{i+1}): \text{$(p_0,\cdots p_N)$
        joins $p$ to $q$}
     \right\}.
   \end{equation}
   Thus $\proj \bar n_k+1,\bar n_k.$ becomes open, cellular and
   $1$-Lipschitz. Moreover for each $\covsheet\in\ccover X_{\bar
     n_k+1}.$ the map $\proj \bar n_k+1,\bar n_k.:\covsheet\to X_{\bar
     n_k}$ is an isometry. Moreover, letting $\proj \alpha,\beta. =
   \proj
   \alpha,\alpha-1.\circ\proj\alpha-1,\alpha-2.\circ\cdots\circ\proj\beta+1,\beta.$,
   we have that for each $\chrsheet\in\chrom X_{\bar n_k+1}.$ the map
   $\proj\bar n_k+1,\bar n_k.:\chrsheet\to[0,1]^3$ is an isometry.
   Note that:
   \begin{equation}
     \label{eq:constr-inv-lim-sys-11}
     \#\chrom X_{\bar n_k+1}. = 2^{\bar n_k+1},
   \end{equation}
   so that one has the representation:
   \begin{equation}
     \label{eq:constr-inv-lim-sys-12}
     \mu_{X_{\bar n_k+1}}=2^{-\bar n_k-1}\sum_{\chrsheet\in\chrom
       X_{\bar n_k+1}.}\lebmeas 3.\on\chrsheet.
   \end{equation}
   We also define the set of jump pairs:
   \begin{equation}
     \label{eq:constr-inv-lim-sys-13}
     \jumpp X_{\bar n_k+1}. = \bigcup_{Q\in\todoub X_{\bar
         n_k}.}\jumpp Q..
   \end{equation}
   \par An important difference is that for $l\le \bar n_k$, if
   $\{(p,q)\}\in \jumpp X_l.$ both $p$ and $q$ are going to be
   centers of some $Q\in\todoub X_{\bar n_k}.$. Thus $p$ gets replaced
   by a pair $\{p\greendec, p\reddec\}$ and $q$ gets replaced by 
   $\{q\greendec, q\reddec\}$. Moreover, for all choices
   $\alpha,\beta\in\{\text{green}, \text{red}\}$ we have:
   \begin{equation}
     \label{eq:constr-inv-lim-sys-14}
     d_{X_{\bar n_k+1}}(p_\alpha, q_\beta) = d_{X_{\bar n_k}}(p,q).
   \end{equation}
   Finally let
   \begin{equation}
     \label{eq:constr-inv-lim-sys-15}
     \gates X_{\bar n_k+1}. = \bigcup_{Q\in\todoub X_{\bar
         n_k}.}\gates Q..
   \end{equation}
   For $\bar n_k+1\le l < \bar n_{k+1}$ we explain how to construct
   $X_{l+1}$ from $X_l$. Let $\todoub X_l.=\cell X_l.\setminus\gates
   X_l.$ and for each $Q\in\todoub X_l.$ replace it with $\dym
   n_k,Q.$; on the other hand, subdivide each $Q\in\gates X_l.$ into
   smaller subcubes so that all the cells of $\cell X_{l+1}.$ have the
   same side length $\slen(X_{l+1})$; the set of cells which were only
   subdivided will be denoted by $\subdiv X_l.$. We then construct
   $d_{X_{l+1}}$, $\proj l+1,l.$, etc\dots as we did for $X_2$ (but we
   replace $n_1$ with $n_k$). Finally let:
   \begin{equation}
     \label{eq:constr-inv-lim-sys-16}
     \gates X_{l+1}. = \bigcup_{Q\in\todoub X_l.}\gates
     Q. \cup\bigcup_{
       \substack{Q\in\cell X_{l+1}.\\
         Q\subset\gates X_l.}}Q.
   \end{equation}
   \par Let $\{X_l\}_l$, $\{\mu_l\}_l$ and $\{\proj l+1,l.\}_l$ denote the resulting
     inverse system; let $X_\infty$, denote the inverse limit;
     then~(\ref{eq:GH-cond-inv-sys}) holds (compare the discussion in
     Lemma~\ref{lem:shape-balls}) and thus $X_l\to X_\infty$ in the
     Gromov-Hausdorff sense and we also obtain a limit measure
     $\mu_\infty$ such that $\proj\infty,l\#.\mu_\infty=\mu_l$.
     Here we summarize three properties of
     the inverse system:
     \begin{description}
     \item[(IsoCov)] For each $\covsheet\in\ccover X_{l+1}.$ the map
       $\proj l+1,l.:\covsheet\to X_l$ is an isometry.
     \item[(IsoChrom)] For each $\chrsheet\in\chrom X_{l+1}.$ the map
       $\proj l+1,l.:\chrsheet\to[0,1]^3$ is an isometry.
     \item[(MuChrom)] One can represent $\mu_{X_{l+1}}$ as:
       \begin{equation}
         \label{eq:constr-inv-lim-sys-17}
         \mu_{X_{l+1}}=2^{-l-1}\sum_{\chrsheet\in\chrom
           X_{l+1}.}\lebmeas 3.\on\chrsheet.
       \end{equation}
     \end{description}
   \end{constr}
\begin{rem}[Chromatic Labels]
  \label{rem:chromatic-labels}
  Note that we can assign to chromatic sheets in
  $X_l$ a color label of length $l$ consisting of entries which are
  either green or red and that in passing to $X_{s+1}$ each sheet gets
  doubled: we can either append red or green to the label. We can then
  induce a chromatic label also on points; some points belong to
  just one chromatic sheet and so the label is unambiguous; for points
  belonging to more than one sheet we make all possible labelings valid.
\end{rem}
\subsection{Bounded local geometry and the horizontal gradient}
\MyOutline{
\label{subsec:local-geo}
\begin{enumerate}
\item A discussion on balls $\checkmark$
\item The fundamental $(\varepsilon r)$-discrete configuration $\checkmark$
\item Verification of doubling $\checkmark$
\item The module of derivations $\checkmark$
\end{enumerate}
}
\begin{defn}[Discrete logarithm]
  \label{defn:discrete-log}
  For $r\in(0,1/2]$ let $\lg(r)$ be the integer such that:
  \begin{equation}
    \label{eq:discrete-log-1}
    \slen(X_{\lg(r)+1}) \le r < \slen(X_{\lg(r)}).
  \end{equation}
  For $\varepsilon\in(0,1]$, as
  $\slen(X_{l+1})\le\frac{1}{2}\slen(X_l)$, one has the estimate:
  \begin{equation}
    \label{eq:discrete-log-2}
    \lg(\varepsilon r) - \lg(r) \le \log_2\frac{1}{\varepsilon} + 5.
  \end{equation}
\end{defn}
\begin{lem}[Shape of balls]
  \label{lem:shape-balls}
  Let $j>l$ ($j=\infty$ being admissible), $p_j\in X_j$ and $p_l =
  \proj j,l.(p_j)$. Then:
  \begin{equation}
    \label{eq:shape-balls-s1}
    \uball p_j,r,X_j. \subset \projinv j,l.(\uball p_l,r,X_l.) \subset
    \uball p_j, r + 4\times\slen(X_{l}),X_j..
  \end{equation}
  There is a universal constant $C$ such that for each $k\in\natural$,
  $p_k\in X_k$ and $r\in[0,\sqrt{3}]$ one can control the number of
  chromatic sheets intersected by $\uball p_k,r,X_k.$ as follows:
  \begin{equation}
    \label{eq:shape-balls-s2}
    \#\{\chrsheet\in\chrom X_k.: \chrsheet\cap\uball
    p_k,r,X_k.\ne\emptyset\} \le C 2^{k-\lg(r)}.
  \end{equation}
\end{lem}
\begin{proof}
  As $\proj j,l.$ is $1$-Lipschitz and open, $\proj j,l.: \uball
  p_j,r,X_j.\to \uball p_l,r,X_l.$ is surjective, which gives the
  inclusion:
  \begin{equation}
    \label{eq:shape-balls-p3}
    \uball p_j,r,X_j. \subset\projinv j,l.(\uball p_l,r,X_l.).
  \end{equation}
  Let $q\in\projinv j,l.(\uball p_l,r,X_l.)\setminus \uball
  p_j,r,X_j.$; then we can reach $q$ from $\uball p_j,r,X_j.$ by
  changing chromatic sheets by looking at the colors added between
  $X_{l+1}$ and $X_j$ (for $j=\infty$ one should use a limiting
  argument). In fact, note that we can assign to chromatic sheets in
  $X_s$ a color label of length $s$ consisting of entries which are
  either green or red and that in passing to $X_{s+1}$ each sheet gets
  doubled: we can either append red or green to the label. Suppose now
  that in $X_s$ we are on a chromatic sheet $\chrsheet$ and we want to
  move to the chromatic sheet that differs only on the last entry of
  the color label. We just need to move to a gate $\gates X_{l+1}.$
  and this can be accomplished by traveling a distance at most
  $\slen(X_{l})$. As $\slen(X_{s+1})\le\frac{1}{2}\slen(X_s)$ we can
  control, via a geometric series, also the total distance to travel to change colors added between
  $X_s$ and $X_{s+1}$ so we get the inclusion:
  \begin{equation}
    \label{eq:shape-balls-p4}
    \projinv j,l.(\uball p_l,r,X_l.)\subset\uball p_j, r + 4\slen(X_l),X_j..
  \end{equation}
  We now pass to the bound on the number of chromatic sheets. Let
  $q\in X_k\cap\uball p_k,r,X_k.$ and for $s\le k$ let $\chrsheet_s(q)$ be a chromatic
  sheet of $X_s$ containing $\proj k,s.(q)$. If for some $s\le k$
  there is no choice of $\chrsheet_s(q)$ such that $\chrsheet_s(q)$
  passes through $\proj k,s.(p_k)$, choose $s=s(q)$ as small as
  possible having this property. Then either $d(q,p_k)\ge \slen(X_{s(q)})$
  or $\proj k,s(q).(q)$ lies at distance $<\slen(X_{s(q)})$ from
 $\partial K$, where $K$ is the central cube 
 that gets doubled in passing from $Q$ to $\dym n=n(s),Q.$ for
 $Q\in\todoub X_{s-1}.$. Consider the set of those $s(q)$ such that
 the second case happens and $d(q,p_k)<\slen(X_{s(q)})/16$:
  this can only happen for one value $\tilde s(q)$ of $s(q)$ because
  $\partial K$ is a a distance $>1/8\slen(X_{s-1})$ from $\partial Q$.
  Now for each $l$ such that $r\ge\slen(X_l)/16$ we can change at most two colors
  in the chromatic sheet and hence the bound~(\ref{eq:shape-balls-s2}) follows.
\end{proof}
\MyOutline{
\begin{macrotest}
  $\disc,.$ $\disc a,.$ $\disc a,b.$ $\lg(r)$ $\grid 5.$ $\fund p,,.$
  $\fund p,1,.$ $\fund p,,2.$ $\fund p,1,2.$ $\slen(A)$ $\fundsub p, ,a,a.$
\end{macrotest}
}
\begin{defn}[The fundamental $(\varepsilon, r)$-configuration]
  \label{defn:fund-epsi-r-conf}
  Fix $\varepsilon\in(0,1/400)$ and let
  $\disc,.=\{j\varepsilon^2r\}_{j=1}^{\lceil
    1/\varepsilon^2\rceil}$. Pick $p_\infty\in X_\infty$ and
  $r\in(0,1/2)$. Let $p_k=\proj \infty,k.(p_\infty)$ and let:
  \begin{equation}
    \label{eq:fund-epsi-r-conf-1}
    \grid p_0.=\{p_0 + ((-1)^{h_1}\varrho_1, (-1)^{h_2}\varrho_2, (-1)^{h_3}\varrho_3)\}_{
      \substack{\varrho_i\in\disc,.\\
        h_i\in\{0,1\}}}.
  \end{equation}
  Then $\grid p_0.\subset\uball p_0,\sqrt{3}(1+\varepsilon)r,X_0.$ and
  $\grid p_0.$ is $\varepsilon^2r$-dense in $\uball p_0,r,X_0.$. Let
  $j_0 = \lg(\varepsilon^2r)$ and denote by $\chrom j_0.$ the set of
  chromatic sheets which intersect $\uball p_{j_0},r,X_{j_0}.$. For
  each $\chrsheet\in\chrom j_0.$ let
  $\grid\chrsheet.=\chrsheet\cap(\proj j_0,0.|\chrsheet)^{-1}(\grid
  p_0.)$. Then set:
  \begin{equation}
    \label{eq:fund-epsi-r-conf-2}
    \grid p_{j_0},r.=\bigcup_{\chrsheet\in\chrom j_0.}\grid\chrsheet.,
  \end{equation}
  which is $\varepsilon^2r$-dense in $\uball p_{j_0},r,X_{j_0}.$ and
  lies in $\uball p_{j_0},2\sqrt{3}r,X_{j_0}.$.
  \par For $j>j_0$ define $\chrom j_0.$ as follows; for each
  $\chrsheet\in\chrom j-1.$ choose just the green lift, i.e.~take
  $\tilde \chrsheet\in\chrom j.$ such that $\proj
  j,j-1.(\tilde\chrsheet)=\chrsheet$ and the label of
  $\tilde\chrsheet$ is obtained by appending green to that of
  $\chrsheet$. As $\proj j,0.:\tilde\chrsheet\to[0,1]^3$ is an
  isometry, let $\grid\tilde\chrsheet.=\tilde\chrsheet\cap(\proj
  j,0.|\tilde\chrsheet)^{-1}(\grid p_0.)$. Finally let:
  \begin{equation}
    \label{eq:fund-epsi-r-conf-3}
    \grid p_j,r. = \bigcup_{\chrsheet\in\chrom j.}\grid\chrsheet.,
  \end{equation}
  which lies in $\uball p_j,2\sqrt{3}r,X_j.$. 
\par We now show that $\grid
  p_j,r.$ is $(5\varepsilon r)$-dense in $\uball p_j,r,X_j.$. From a
  point $q\in\chrsheet$ to change color labels at the positions $s > j_0$
  one needs to travel a distance at most $4\slen(X_{j_0})\le
  4\varepsilon^2r<\varepsilon r$. As $\grid p_0.$ is $\varepsilon
  r$-dense in $\uball p_0,r,X_0.$ we conclude that $\grid p_j,r.$ is
  $(5\varepsilon r)$-dense in $\uball p_j,r,X_j.$. Finally
  by~(\ref{eq:shape-balls-s1}) there is a uniform bound
  $C=C(\varepsilon)$ on the cardinality of $\grid p_j,r.$. We will
  call $\grid p_j,r.$ \textbf{a fundamental configuration at $p_j$,
  at scale $r$ and resolution $\varepsilon$}: we will denote it by
  $\fundsub p_j,,,X_j.$.
  \par Finally, we can obtain $\fundsub p_\infty,,, X_\infty.$ by a
  limiting procedure. In fact, each $\chrsheet\in\chrom j_0.$ gives
  rise to a sequence fo chromatic sheets $\chrsheet^{(j_0)}=\chrsheet,
  \chrsheet^{(j_0+1)}, \cdots$ where we keep appending green to the
  labels. This yields a limit sheet $\chrsheet^{(\infty)}\subset
  X_\infty$ and we can then let  $\grid\chrsheet^{(\infty)}.=\chrsheet^{(\infty)}\cap(\proj
  \infty,0.|\chrsheet^{(\infty)})^{-1}(\grid p_0.)$ and then proceed
  as above.
\end{defn}
\begin{lem}[The measures are doubling]
  \label{lem:ver-doubling}
  The measures $\mu_j$ ($j=\infty$ being admissible) are uniformly
  doubling, i.e.~there is a universal constant $C$ such that for each
  $j\in\natural\cup\{\infty\}$, $p_j\in X_j$ and $r\in(0,\sqrt{3})$
  one has:
  \begin{equation}
    \label{eq:ver-doubling-s1}
    \mu_j\biggl(
    \uball p_j,r,X_j.
    \biggr) \le C\mu_j\biggl(
      \uball p_j,r/2,X_j.
    \biggr).
  \end{equation}
\end{lem}
\begin{proof}
We treat the case $j<\infty$ as then $j=\infty$ follows by a limiting
argument. Combining~(\ref{eq:shape-balls-s2})
with~(\ref{eq:constr-inv-lim-sys-12}) we deduce:
\begin{equation}
  \label{eq:ver-doubling-p1}
  \mu_j\biggl(
    \uball p_j,r,X_j.
  \biggr)\le\frac{4\pi}{3}C r^3 2^{-\lg(r)},
\end{equation}
$C$ being the constant from~(\ref{eq:shape-balls-s2}).
\par Fix a chromatic sheet $\chrsheet$ containing $p_j$; by traveling
a distance $\le r/4$ we can change all the last $j-\lg(r/16)\ge
j-\lg(r)-9$ entries of the color label of $\chrsheet$. Thus:
\begin{equation}
  \label{eq:ver-doubling-p2}
  \bigcup_{\tilde\chrsheet\in\mathcal{S}}\tilde\chrsheet\cap\proj
  j,0.^{-1}(\uball {\proj j,0.(p_0)},r/4,X_0.)\subset\uball p_j,r/2,X_j.,
\end{equation}
where $\mathcal{S}\subset\chrom X_j.$ has cardinality at least
$j-\lg(r)-9$. We then have:
\begin{equation}
  \label{eq:ver-doubling-p3}
  \mu_j\left(
    \uball p_j,r/2,X_j.
    \right)\ge\frac{4\pi}{3}\frac{r^3}{64}2^{-\lg(r)-9}.
\end{equation}
\end{proof}
\begin{defn}[The horizontal gradient]
  \label{defn:hz-grad}
  We want to describe the horizontal gradient $\nabla$ in $X_l$ for
  $l\in\natural\cup\{0,\infty\}$. For $l=0$ we just take the usual
  gradient as $X_0=[0,1]^3$. In general, for $l<\infty$ the measure
  $\mu_l$ has a \textbf{$3$-rectifiable representation}, i.e.~it can
  be represented as an integral of measures associated to
  $3$-rectifiable sets:
  \begin{equation}
    \label{eq:hz-grad-1}
    \mu_{X_l}=2^{-l}\sum_{\chrsheet\in\chrom X_l.}\lebmeas 3.\on\chrsheet;
  \end{equation}
  as each $\chrsheet\in\chrom X_l.$ can be identified with $[0,1]^3$ we
  can take the standard gradient $\nabla$ on each $\chrsheet$ and
  obtain the horizontal gradient $\nabla$ on $X_l$.
  \par Let $\vec x$ be the tuple $(x^1,x^2,x^3)$ of coordinate
  functions on $[0,1]^3$; with abuse of notation we will also write
  $\vec x$ for $\vec x\circ\proj l,0.$; then at each $p\in X_l$ one
  has $\nabla \vec x(p) = \text{Id}_{\real^3}$.
  \par For $l=\infty$ let $\chrom X_\infty.$ denote the set of all
  the sequences $(\chrsheet_i)_{i=0}^\infty$ where $\chrsheet_i\in\chrom
  X_i.$ and $\proj i+1,i.(\chrsheet_{i+1})=\chrsheet_i$ for each
  $i$. Then $(\chrsheet_i)_{i=0}^\infty$ admits an inverse limit
  $\chrsheet_\infty$, and note that $\chrsheet_\infty$ also completely
  determines the sequence $(\chrsheet_i)_{i=0}^\infty$ letting
  $\chrsheet_i=\proj \infty,i.(\chrsheet_\infty)$. The uniform probability
  measures $P_l=2^{-l}$ on $\chrom X_{l}.$ pass to the limit to a
  probability measure $P_\infty$ on $\chrom X_\infty.$. More
  concretely, using sequences on green and red, we can identify
  $\chrom X_\infty.$ with the standard Cantor set and $P_\infty$
  becomes the corresponding standard probability measure. Taking the
  limit in~(\ref{eq:hz-grad-1}) we get:
  \begin{equation}
    \label{eq:hz-grad-2}
    \mu_\infty = \int_{\chrom X_\infty.}\lebmeas 3.\on \chrsheet\,dP_\infty(\chrsheet);
  \end{equation}
  on each $\chrsheet$ the operator $\nabla$ is well-defined, and
  thanks to~(\ref{eq:hz-grad-2}) we can combine them to obtain the
  horizontal derivative on $X_\infty$. Note also that $\nabla \vec x =
  \text{Id}_{\real^3}$ on $X_\infty$ where with abuse of notation we
  have written $\vec x$ for $\vec x\circ\proj\infty,0.$.
\end{defn}
\subsection{Horizontal paths with jumps}
\MyOutline{
\label{subsec:hz-paths}
\begin{enumerate}
\item Definition $\checkmark$
\item Existence of horizontal paths with jumps $\checkmark$
\end{enumerate}
}
\begin{defn}[Horizontal paths]
  \label{defn:hz-paths}
  A \textbf{horizontal segment} $\sigma$ in $X_j$ ($j=\infty$ being
  admissible) is a geodesic segment such that $\proj j,0.(\sigma)$ is
  a segment of $X_0$ parallel to one of the coordinate axes of
  $[0,1]^3$. We allow for a segment to be degenerate, i.e. to be~just a point.
  \par A \textbf{horizontal path} $\ph$ in $X_j$ is a finite tuple
  $\ph=(\sigma_1,\cdots,\sigma_N)$ of horizontal segments such that
  for $1\le i < N$ the end point of $\sigma_i$ is the starting point
  of $\sigma_{i+1}$. The \textbf{length} of $\ph$ is the sum of the
  lengths of its segments:
  \begin{equation}
    \label{eq:hz-paths-1}
    \len(\ph) = \sum_{i=1}^N\len(\sigma_i).
  \end{equation}
\end{defn}
\begin{defn}[The set of total jump pairs]
  \label{defn:tot-jump-pairs}
  The set $\tjumpp X_j.$ of \textbf{total jump pairs} of $X_j$
  ($j=\infty$ being admissible) consists of all $\{q,q'\}\subset X_j$
  such that:
  \begin{enumerate}
  \item Either $j<\infty$ and one has $(q,q')\in\jumpp X_j.$; in this
    case $\gates \{q,q'\}.$ is the union of the two gates of $X_j$
    containing $q,q'$.
  \item Or there is an $l<j$ such that $(\proj j,l.(q), \proj
    j,l.(q'))\in\jumpp X_l.$; in this case $\gates \{q,q'\}.$ is the
  union of the two gates of $X_l$ containing $\proj j,l.(q),\proj j,l.(q')$.
  \end{enumerate}
  In this second case note that because
  of~(\ref{eq:constr-inv-lim-sys-9})
  \begin{equation}
    \label{eq:tot-jump-pairs-1}
    d_{X_j}(q,q')=d_{X_l}(\proj j,l.(q), \proj j,l.(q')).
  \end{equation}
\end{defn}
\begin{defn}[Horizontal paths with jumps]
  \label{defn:hz-paths-jumps}
  A \textbf{horizontal path with jumps} $\jph$ in $X_j$ ($j=\infty$
  being admissible) is a finite alternating tuple
  $(\ph_1,\jp_1,\ph_2,\cdots,\jp_{N-1},\ph_N)$ where:
  \begin{enumerate}
  \item The $\{\ph_i\}_{i=1}^N$ are horizontal paths and
    $\{\jp_i\}_{i=1}^N\subset\tjumpp X_j.$.
  \item For $1\le i\le N-1$ $\jp_i$ consists of the end point of
    $\ph_i$ and the starting point of $\ph_{i+1}$.
  \end{enumerate}
  The \textbf{length} of $\jph$ is:
  \begin{equation}
    \label{eq:hz-paths-jumps-1}
    \len(\jph) = \sum_{i=1}^N\len(\ph_i) +
    \sum_{i=1}^{N-1}\{d_{X_j}(q,q'): (q,q')=\jp_i\}.
  \end{equation}
\end{defn}
\begin{lem}[Existence of good horizontal paths with jumps]
  \label{lem:good-hz-paths}
  Let $\fundsub p_l,,,X_l.$ be a fundamental configuration in $X_l$
  ($l=\infty$ being admissible). Then there is a universal constant
  $C$ independent of $l$, $p_l$, $r$ and $\varepsilon$ such that for
  each $q_l\in\fundsub p_l,,,X_l.$ there is either a horizontal path
  $\gamma=\ph$ or a horizontal path with jumps $\gamma=\jph$ such
  that:
  \begin{description}
  \item[(Gd1)] $\gamma$ starts at $p_l$ and ends at $q_l$.
  \item[(Gd2)] If $\gamma=\jph$ there is only one jump,
    i.e.~$\gamma=(\ph_{-},\jp,\ph_{+})$. 
  \item[(Gd3)] $\len(\gamma)\le C d_{X_l}(p_l,q_l)$.
  \item[(Gd4)] With the exception of at most $10$ horizontal segments
    in $\gamma$, for each other horizontal segment $\sigma$ one has:
    \begin{equation}
      \label{eq:good-hz-paths-s1}
      \len(\sigma)\ge\frac{\varepsilon^3r}{400}.
    \end{equation}
  \item[(Gd5)] $\gamma$ contains at most $15$ horizontal segments.
  \end{description}
\end{lem}
\begin{proof}
  The construction will be inductive; for $j<l$ let $p_j=\proj
  l,j.(p_l)$ and $q_j=\proj l,j.(q_l)$.
  \par\noindent\texttt{Step 1: The construction in $X_0$ and $X_1$}
  \par $X_0$ is just $[0,1]^3$ with the Euclidean metric and we know
  that for $i\in\{1,2,3\}$ either $x^i(p_0)=x^i(q_0)$ or
  $|x^i(p_0)-x^i(q_0)|\ge\varepsilon^2 r$. We can then find a
  horizontal path $\gamma_0=\ph_0$ starting at $p_0$, ending at $q_0$
  and satisfying the following conditions (henceforth referred to as
  \textbf{(Inv1)}):
  \begin{description}
  \item[(Inv1:1)] $\gamma_0$ consists of at most $3$ horizontal paths
    and $\len(\gamma_0)\le 3d_{X_0}(p_0,q_0)$.
  \item[(Inv1:2)] The length of each horizontal segment in $\gamma_0$
    is $\ge\varepsilon^3 r$.
  \end{description}
  \par As $\proj 1,0.$ is open we can lift $\gamma_0$: consider the
  set $\lift\gamma_0.$ of all lifts of $\gamma_0$ (i.e.~horizontal
  paths $\gamma_1$ that are mapped to $\gamma_0$ by $\proj 1,0.$)
  starting at $p_1$. If one such lift $\ph_1$ ends at $q_1$ let
  $\gamma_1=\ph_1$ and note that it will satisfy \textbf{(Inv1)}
  (change the subscripts in \textbf{(Inv1:1--2)} from $0$ to $1$).
  \par Assume that this does not happen. As in
  Construction~\ref{constr:nqc-diamond} let $K$ be the subcube that
  gets doubled and $\{c\greendec, c\reddec\}$ the pair of jump
  points. Then $p_1$ and $q_1$ belong to lifts of $K$ lying in
  chromatic sheets with different colors. Note that $\partial K$ can
  be regarded also as a subset of $X_1$ (the metrics $d_{X_1}$ and
  $d_{X_0}$ agree on it).
  \par We first consider the case in which there is a $p_K\in\partial
  K$ such that:
  \begin{equation}
    \label{eq:good-hz-paths-p1}
    d_{X_1}(p_1,p_K) + d_{X_1}(p_K,q_1)\le 8d_{X_1}(p_1,q_1).
  \end{equation}
  Consider the 6 quantities:
  \begin{equation}
    \label{eq:good-hz-paths-p2}
    \biggl\{\bigl|x^i(p_1)-x^i(p_K)\bigr|,
    \bigl|x^i(p_K)-x^i(q_1)\bigr|\biggr\}_{i\in\{1,2,3\}};
  \end{equation}
  as $d_{X_1}(p_1,q_1)\ge\varepsilon^2 r$ at most 5 of the above
  quantities can be $\le \varepsilon^3 r/16$ and thus we can find a
  horizontal path $\gamma_1=\ph_1$ starting at $p_1$ and ending at
  $q_1$ and satisfying  the following conditions (henceforth referred to as
  \textbf{(Inv2)}):
  \begin{description}
  \item[(Inv2:1)] $\gamma_1=(\ph_{-},\ph_{+})$, $\ph_{-}$ joins $p_1$
    to $p_K$ and $\len(\ph_{-})\le 3d_{X_1}(p_1,p_K)$; $\ph_{+}$ joins
    $p_K$ to $q_1$ and $\len(\ph_{+})\le 3d_{X_1}(p_K,q_1)$; as a
    consequence $\len(\gamma_1)\le 24d_{X_1}(p_1,q_1)$.
  \item[(Inv2:2)] With the exception of at most 5 horizontal segments
    in $\gamma_1$, for each other horizontal segment $\sigma$ one has
    $\len(\sigma)\ge\varepsilon^3 r/16$.
  \end{description}
  \par If such a $p_K$ does not exist to change color we must use the
  jump pair $(c\greendec, c\reddec)$; moreover, without loss of generality we
  can assume that:
  \begin{equation}
    \label{eq:good-hz-paths-p3}
    d_{X_1}(p_1, c\reddec) + d_{X_1}(c\reddec,c\greendec) +
    d_{X_1}(c\greendec,q_1) = d_{X_1}(p_1,q_1). 
  \end{equation}
  We can then find a horizontal path with jumps
  $\gamma_1=\jph_1=(\ph_{-},\jp,\ph_{+})$ which satisfies the following conditions (henceforth referred to as
  \textbf{(Inv3)}):
  \begin{description}
  \item[(Inv3:1)] $\ph_{-}$ joins $p_1$ to $c\reddec$, $\jp =
    (c\greendec,c\reddec)$, $\ph_{+}$ joins $c\greendec$ to $q_1$;
    moreover $\len(\ph_{-})\le 3d_{X_1}(p_1,c\reddec)$,
    $\len(\ph_{+})\le 3d_{X_1}(c\greendec,q_1)$ and thus
    $\len(\gamma_1)\le 3d_{X_1}(p_1,q_1)$.
  \item[(Inv3:2)] Except for at most 6 horizontal segments in
    $\gamma_1$,  for each other horizontal segment $\sigma$ one has
    $\len(\sigma)\ge\varepsilon^3 r/16$.
  \end{description}
  For the record we also note that:
  \begin{equation}
    \label{eq:good-hz-paths-p4}
    d_{X_1}(p_1,q_1)\ge 4\slen(X_1).
  \end{equation}
  \par\noindent\texttt{Step 2: The construction in $X_2$.}
  \par Consider the lifts $\lift\gamma_1.$ of $\gamma_1$ in $X_2$
  starting at $p_2$. If one such lift $\tilde\gamma$ ends at $q_2$ let
  $\gamma_2=\tilde\gamma$, which will satisfy the same of
  \textbf{(Inv1--3)} that $\gamma_1$ satisfied.
  \par Supose that there is no such a lift. Then there are
  $Q_1,Q_2\in\todoub X_1.$ such that $p_1\in Q_1$, $q_1\in Q_2$. Let
  $K_i$ be the central subcube of $Q_i$ which is doubled in
  constructing $\dym n_1,Q_i.$; then $p_1\in K_1$ and $q_1\in K_2$ and
  $p_2$, $q_2$ have different color labels at position 2. If
  $\gamma_1$ satisfied \textbf{(Inv2)} we whould have crossed
  $\partial K_1$ following $\gamma_1$ and so we would have been able
  to find a lift that changed the second color label and ended at
  $q_2$. If $\gamma_1$ satisfied \textbf{(Inv3)}, to reach $c\reddec$
  in the lift of $K$, we would again have crossed $\partial K_1$ and
  we would have been able to change the second color label, finding a
  lift ending at $q_2$. We thus conclude that $\gamma_1$ satisfies
  \textbf{(Inv1)}. Now we can essentially argue as before. One
  possibility is that there was a $p_K\in\partial K_1\cup\partial K_2$
  (now regarded as a subset of $X_2$) such that:
  \begin{equation}
    \label{eq:good-hz-paths-p5}
    d_{X_2}(p_1,p_K) + d_{X_2}(p_K,q_1)\le 8d_{X_2}(p_2,q_2).
  \end{equation}
  In this case, as in \texttt{Step 1} we can produce a $\gamma_2$
  starting at $p_2$ and ending at $q_2$ which satisfies
  \textbf{(Inv2)}. Otherwise, we can use either the lifts of the
  center of $K_1$ or $K_2$ to change the second color label and argue
  as in \texttt{Step 1} to produce a $\gamma_2$ satisfying
  \textbf{(Inv3)}.
  \par\noindent\texttt{Step 3: The construction in $X_l$ for
    $2<l<\infty$.}
  \par Consider the lifts $\lift\gamma_{l-1}.$ of $\gamma_{l-1}$ in
  $X_l$ starting at $p_l$. If one such a lift $\tilde\gamma$ ends at
  $q_l$, let $\gamma_l=\tilde\gamma$ which will satisfy the same
  \textbf{(Inv1--3)} that $\gamma_l$ satisfied.
  \par Assume that this is not the case. Then there are
  $Q_1,Q_2\in\todoub X_{l-1}.$ such that, denoting by $K_1,K_2$ the
  central subcubes to be doubled, $p_{l-1}\in K_1$ and $q_{l-1}\in
  K_2$, and the color labels of $p_{l-1}$ and $q_{l-1}$ differ at the
  $l$-th position. Assume that $\gamma_{l-1}$ satisfied
  \textbf{(Inv2)} and that the transition from \textbf{(Inv1)} to
  \textbf{(Inv2)} (note that $\gamma_0$ will always satisfy
  \textbf{(Inv1)}) occurred at level $j$. Let then $K$ denote the
  central subcube in $X_{j-1}$ whose boundary was used by $\gamma_j$
  to change the color label at position $j$. Then to reach $\partial
  K$ (which can be regarded as a subset of $X_{l-1}$) from $p_{l-1}$ we
  would have crossed $\partial K_1$, so we would have been able to
  lift $\gamma_{l-1}$ to end at $q_l$. Assume that $\gamma_{l-1}$
  satisfied invariant \textbf{(Inv3)}. Let $\{c_1,c_2\}$ denote the
  couple of jump points used in $\gamma_{l-1}$. Then there are two
  cases to consider. One is that $\{c_1,c_2\}$ can be lifted also to
  change the $l$-th color label. Recalling \texttt{Step 2} of
  Construction~\ref{constr:inv-lim-sys} this can happen in
  constructing $X_{\bar n_k+1}$ because all $Q\in\cell X_{\bar n_k}.$
  get replaced by $\dym n_{k+1},Q.$. If this is not the case, then to
  reach $c_1$ or $c_2$ we have to cross $\partial K_1\cup\partial K_2$
  and so can change the $l$-th color label. This implies that a lift
  of $\gamma_{l-1}$ ending at $q_l$ has to exist. We thus conclude that
  $\gamma_{l-1}$ has to satisfy \textbf{(Inv1)} and we can argue as in
  \texttt{Step 2} finding $\gamma_l$ satisfying either \textbf{(Inv2)}
  or \textbf{(Inv3)}.
  \par\noindent\texttt{Step 4: The case $l=\infty$.}
  \par For $l=\infty$ we use a limiting argument. Consider the
  sequence $\{\gamma_j\}$: after some $j_0$ all the $\gamma_j$'s have to
  satisfy the same \textbf{(Inv1--3)}. Then for $j> j_0$, $\gamma_j$
  is a lift of $\gamma_{j+1}$ and we can obtain $\gamma_\infty$ as the
  inverse limit of the $\{\gamma_j\}_{j\ge j_0}$.
\end{proof}
\subsection{Existence of disconnected tangents}
\label{subsec:disc-tangents}
\begin{thm}[Existence of disconnected tangents]
  \label{thm:disc-tangs}
  At $\mu_\infty$-a.e.~$p_\infty\in X_\infty$ there is a
  tangent/blow-up which is not topologically connected.
\end{thm}
\begin{proof}
  For each $p_\infty\in X_\infty$ and $l<\infty$ let
  $\proj\infty,l.(p_\infty)=p_l$. Recall that $\mu_\infty$ is a
  probability measure and define:
  \begin{equation}
    \label{eq:disc-tangs-p1}
    E_k = \biggl\{
    p_\infty: \text{for some $l\in\{\bar n_k+1,\cdots,\bar n_{k+1}\}$
      $p_l\in\gates X_l.$}
    \biggr\}.
  \end{equation}
  In \texttt{Step 2} of Construction~\ref{constr:inv-lim-sys} we
  replaced each $Q\in\cell X_{\bar n_k+1}.$ with $\dym n_{k+1},Q.$:
  this implies that the events $E_k$ and $E_{k+j}$ are independent for
  $j\ge 1$. Moreover, there is a universal constant $c>0$ such that
  for each $k$ one has:
  \begin{equation}
    \label{eq:disc-tangs-p2}
    \mu_{\bar n_k+1}(\gates X_{\bar n_k+1}.)\ge\frac{c}{n_{k+1}^3};
  \end{equation}
  consider now $l$ such that $\bar n_k+1<l\le \bar n_{k+1}$: as in
  Construction~\ref{constr:inv-lim-sys} we do not apply the diamond
  construction to the gates at the previous levels, we have
  the estimate:
  \begin{equation}
    \label{eq:disc-tangs-p3}
    \mu_l(\gates X_l.)\ge\frac{c}{n_{k+1}^3}\mu_{l}\biggl(
    	X_l\setminus\bigcup_{j=\bar n_k+1}^{l-1}\projinv l,j.(\gates X_j.)
    \biggr).
  \end{equation}
  We can therefore estimate the measure of $E_k$ from below:
  \begin{equation}
    \label{eq:disc-tangs-p4}
    \begin{split}
    \mu_\infty(E_k)&\ge\frac{c}{n_{k+1}^3}\sum_{t=0}^{n_{k+1}^3-1}\biggl(
    	1-\frac{c}{n^3_{k+1}}
        \biggr)^t\\ &= 1 - \biggl(
		1-\frac{c}{n^3_{k+1}}
      \biggr)^{n^3_{k+1}}\\ &\ge 1-2e^{-c},
    \end{split}
  \end{equation}
  for $k$ sufficiently large. Hence there is a uniform lower bound on
  $\mu_\infty(E_k)$ and by the the Borel-Cantelli Lemma, for
  $\mu_\infty$-a.e.~$p_\infty$ one has $p_\infty\in E_k$ infinitely
  often.
  \par Assume $p_\infty\in E_k$ and let $t_k$ be such that
  $p_{t_k}\in\gates X_{t_k}.$ where $\bar n_k+1\le t_k\le \bar
  n_{k+1}$. Consider the rescaling
  $Y_{t_k}=\frac{1}{\slen(X_{t_k})}X_{t_k}$; then $\uball
  p_{t_k},\sqrt{n_{k+1}},Y_{t_k}.$ has two distinct connected
  components at distance $\ge 1$. For $j>t_{k}$ we can keep lifting
  these components in the rescaling $Y_j=\frac{1}{\slen(X_{t_k})}X_j$
  to conclude that also $\uball p_j,\sqrt{n_{k+1}},Y_j.$ has two
  connected components at distance $\ge 1$. For $k\nearrow\infty$ we
  have $n_{k+1}\nearrow\infty$ and so we have that
  $\biggl(\frac{1}{\slen(X_{t_k})}X_\infty,p_\infty\biggr)$
  subconverges to a metric space having at least two connected
  components at distance $\ge1$.
\end{proof}
\section{Differentiability of real-valued Lipschitz maps}
\label{sec:real-diff}
In this section we prove differentiability for real-valued Lipschitz
maps, Theorem~\ref{thm:real-diff}. The emphasis is to get an estimate to control how fast a
Lipschitz map will collapse together the points appearing in the jump
part of the horizontal paths
with jumps that we constructed in Section~\ref{sec:inv-lim}. The key estimate is given in
Theorem~\ref{thm:gate-collapse}; this result is based on taking
recursive piecewise harmonic approximations of the function,
Definition~\ref{defn:harmonic-appx} and on some elementary PDE~\ref{lem:energy-lwbd}.
\subsection{Harmonic functions}
\label{subsec:harm-funs}
\MyOutline{
\begin{enumerate}
\item Harmonic functions on the cube are Lipschitz in the interior $\checkmark$
\item Lower bound on the energy $\checkmark$
\item Harmonic approximations $\checkmark$
\item An estimate on the rate of collapse of the gates $\checkmark$
\end{enumerate}
}
\begin{lem}[Lipschitz estimate for harmonic functions]
  \label{lem:harm-lip}
  Let $u:U\to\real$ or $l^2$ be harmonic where $U\subset \real^3$ is
  open. Assume that $\ball p_0,r.\subset U$; then there is a universal
  constant $C$, independent of $u, U, p_0$ or $r$, such that $u$ is:
  \begin{equation}
    \label{eq:harm-lip-s1}
    \text{$\frac{C}{r}\|u\|_{L^1(\ball p_0,r.)}$-Lipschitz}
  \end{equation}
in $\ball p_0,r/3.$.
\end{lem}
\begin{proof}
  The case where $u$ is real-valued is well-explained in~\cite[Ch.~2,Thm.~7]{evans-pdebook}; here we explain the minor
  modifications needed for $l^2$-valued harmonic functions. Let $u_j$
  be a component of $u$ and let $p\in\ball p_0,r/3.$; by the mean
  value property:
  \begin{equation}
    \label{eq:harm-lip-p1}
    \partial_i u_j(p) = \av_{\ball p,r/3.}\partial_i u_j\,d\lebmeas
    3.=\frac{\alpha_3}{r}\av_{\partial \ball p,r/3.}u_j\nu_i\,d\hmeas 2.,
  \end{equation}
where we used integration by parts, $\alpha_3$ denotes a universal
constant and $\nu$ is the outer normal to $\partial\ball
p,r/3.$. Then:
\begin{equation}
  \label{eq:harm-lip-p2}
  \sum_{j=1}^\infty |\partial_i u_j(p)|^2\le\bigl(
\frac{\alpha_3}{r}\bigr)^2\sum_{j=1}^\infty\av_{\partial\ball
  p,r/3.}|u_j|^2\,d\hmeas 2.,
\end{equation}
and so:
\begin{equation}
  \label{eq:harm-lip-p3}
  \hilbnrm\partial_iu(p).\le\frac{\alpha_3}{r}\max_{q\in\partial\ball
    p,r/3.}\hilbnrm u(q).;
\end{equation}
then~(\ref{eq:harm-lip-s1}) follows applying the mean-value property
to $q\in\ball q,r/3.\subset\ball p_0,r.$.
\end{proof}
\begin{lem}[Lower bound on the energy]
  \label{lem:energy-lwbd}
  Let $Q$ be a cube with sidelength $\slen(Q)$ and for
  $s\in(0,\frac{\slen(Q)}{6}]$ let $sQ$ denote the cube with the same
  center as $Q$ and with sidelength $s$. Assume that
  $F:\overline{Q\setminus sQ}\to\real$ is continuous and locally
  Lispchitz in $Q\setminus sQ$: for each compact $K$ contained in the
  interior of $Q\setminus sQ$ the restriction $F|K$ is Lipschitz. Then
  there is a universal constant $c\hardec$ (independent of $Q$ and
  $s$) such that if:
  \begin{equation}
    \label{eq:energy-lwbd-s1}
    \left|\av_{\partial (sQ)}F\,d\hmeas 2.\right|\ge\eta
  \end{equation}
  and $F=0$ on $\partial Q$ then:
  \begin{equation}
    \label{eq:energy-lwbd-s2}
    \int_{Q\setminus sQ}\|\nabla F\|_2^2\,d\lebmeas 3.\ge c\hardec\eta^2 s.
  \end{equation}          
\end{lem}
\begin{proof}
  \MakeStep{Step 1: Reducing the problem to balls.}
  \par Up to a translation we can assume that $Q$ is centered at
  $0$. Let:
  \begin{equation}
    \label{eq:energy-lwbd-p1}
    \begin{aligned}
      \Psi : Q&\to \ball 0,\frac{\slen(Q)}{2}.\\
      x&\mapsto\frac{\|x\|_\infty}{\|x\|_2}x.
    \end{aligned}
  \end{equation}
  We compute $d\Psi$ at a generic point $x$ where $x_1\ne x_2\ne x_3$
  and where $\|x\|_\infty=|x_1|$:
  \begin{equation}
    \label{eq:energy-lwbd-p2}
    \partial_j\Psi_i = \frac{\|x\|_\infty}{\|x\|_2}\biggl(
    	\delta_{ij}-\frac{x_ix_j}{\|x\|_2^2}
    \biggr) + \frac{
      \sgn(x_1)\chi_{j=1}
    }{\|x\|_2} x_i;
  \end{equation}
  on the one hand:
  \begin{equation}
    \label{eq:energy-lwbd-p3}
    \biggl|\biggl\langle
    	d\Psi(x), \frac{x}{\|x\|_2}
    \biggr\rangle\biggr| = \biggl\|
    	\frac{\sgn(x_1)}{\|x_2\|_2^2}x_1(x_1,x_2,x_3)
    \biggr\|_2\ge\frac{1}{3};
  \end{equation}
  on the other hand if $v$ is a unit vector orthogonal to $x$:
  \begin{equation}
    \label{eq:energy-lwbd-p4}
    \langle d\Psi(x),v\rangle = \frac{\|x\|_\infty}{\|x\|_2}v +\sgn(x_1)v_1\frac{x}{\|x\|_2}
  \end{equation}
  and thus
  \begin{equation}
    \label{eq:energy-lwbd-p5}
    |\langle d\Psi(x),v\rangle|\ge 1.
  \end{equation}
  We conclude that $\Psi$ is $(1/16,16)$-bi-Lipschitz and maps $Q$
  onto $\ball 0,\frac{\slen(Q)}{2}.$, $sQ$ onto $\ball
  0,\frac{s}{2}.$, $\partial Q$
  onto $\partial\ball 0,\frac{\slen(Q)}{2}.$ and $\partial(sQ)$ onto $\partial\ball
  0,\frac{s}{2}.$. We can thus reduce to the case in which $F:\ball
  0,\frac{\slen(Q)}{2}.\to\real$, $F=0$ for $r=\frac{\slen(Q)}{2}$
  (we are using polar coordinates) and
  \begin{equation}
    \label{eq:energy-lwbd-p6}
    \av_{\partial\ball 0,\frac{s}{2}.}F\,d\hmeas 2.\ge\eta,
  \end{equation}
  up to changing the sign of $F$ and by replacing the original $\eta$
  with $\eta/16^4$.
  \MakeStep{Step 2: Symmetrization}
  \par Let $\omega\in\sphere^2$ and define:
  \begin{equation}
    \label{eq:energy-lwbd-p7}
    \tilde F(r,\omega)=\tilde F(r) =
    \av_{\sphere^2}F(r,\tilde\omega)\,d\hmeas 2.(\tilde\omega).
  \end{equation}
  As $F$ is locally Lipschitz in $\ball 0,\frac{\slen(Q)}{2}.\setminus
  \ball 0, \frac{s}{2}.$ so is $\tilde F$. We show that $\tilde F$ has
  lower energy than $F$ in $\ball 0,\frac{\slen(Q)}{2}.\setminus\ball 0,
        \frac{s}{2}.$, so it suffices to bound the energy of $\tilde
  F$ from below:
  \begin{equation}
    \label{eq:energy-lwbd-p8}
    \begin{split}
      \int_{\ball 0,\frac{\slen(Q)}{2}.\setminus\ball 0,
        \frac{s}{2}.}\|\nabla\tilde F(r,\omega)\|_2^2\,r^2dr d\omega
      &= \int_{s/2}^{\slen(Q)/2}r^2\,dr\int_{\sphere^2}d\omega\biggl(
      			\av_{\sphere^2}\partial_rF(r,\tilde\omega)\,d\tilde\omega
      		\biggr)^2\\
      &\le\int_{s/2}^{\slen(Q)/2}r^2\,dr\int_{\sphere^2}d\omega\av_{\sphere^2}\bigl(
      	\partial_r F(r,\tilde\omega)
        \bigr)^2\,d\tilde\omega\\
      &\le \int_{\ball 0,\frac{\slen(Q)}{2}.\setminus\ball 0,
        \frac{s}{2}.}\|\nabla F(r,\omega)\|_2^2\,r^2dr d\omega.
    \end{split}
  \end{equation}
  Note also that $\tilde F=0$ for $r=\frac{\slen(Q)}{2}$ and $\tilde
  F=a\ge \eta$ for $r=\frac{s}{2}$.
  \MakeStep{Step 3: Comparison with a harmonic function.}
  \par The minimum energy will be attained by the harmonic function
  with the same boundary conditions as $\tilde F$: the general
  solution is of the form $A/r+B$ and we get:
  \begin{equation}
    \label{eq:energy-lwbd-p9}
    \begin{aligned}
      A &= \frac{as}{2(\slen(Q)-s)}\slen(Q)\\
      B &= -\frac{2A}{\slen(Q)}.
    \end{aligned}
  \end{equation}
  We can then compute the energy of this function as follows:
  \begin{equation}
    \label{eq:energy-lwbd-p10}
    \begin{split}
      4\pi \int_{s/2}^{\slen(Q)/2}\biggl(
      	\frac{A}{r^2}
      \biggr)^2r^2\,dr &= 4\pi A^2\biggl[
      	-\frac{1}{r}
      \biggr]_{r=s/2}^{\slen(Q)/2}\\
      &= 4\pi A^2\frac{2(\slen(Q)-s)}{s\slen(Q)}\\
      &= \pi a^2 s\frac{\slen(Q)}{\slen(Q)-s}\\
      &\ge \eta^2s.
    \end{split}
  \end{equation}
\end{proof}
\begin{rem}
  \label{rem:sobolev}
  For the proof of Lemma~\ref{lem:energy-lwbd} we made the simplest
  assumption of $F$ being locally Lipschitz; one might have made a
  more general one to run the same argument, say assuming the $F$
  belonged to the Sobolev space $W^{1,2}(Q\setminus sQ)$ and extended
  continuously to $\overline{Q\setminus sQ}$.
\end{rem}
\MyOutline{
\begin{macrotest}
  $\skeleton,.$ $\skeleton a,.$ $\skeleton ,b.$ $\skeleton a,b.$ 
  $\hskeleton,.$ $\hskeleton a,.$ $\hskeleton,b.$ $\hskeleton a,b.$
  $\hharmonic,.$ $\hharmonic a,.$ $\hharmonic,b.$ $\hharmonic a,b.$
  $\gharmonic,.$ $\gharmonic a,.$ $\gharmonic,b.$ $\gharmonic a,b.$
\end{macrotest}
}
\begin{defn}[Piecewise harmonic approximations]
  \label{defn:harmonic-appx}
  We define the \textbf{$2$-skeleton} of $X_j$ ($j<\infty$) as:
  \begin{equation}
    \label{eq:harmonic-appx-1}
    \skeleton ,.=\bigcup_{Q\in\cell X_j.}\partial Q;
  \end{equation}
  note that for $l\ge j$ ($l=\infty$ being admissible) $\skeleton,.$
  embedds isometrically in $X_l$.
  \par We define the
  \textbf{$2$-harmonic skeleton} of $X_j$ ($j<\infty$) as:
  \begin{equation}
    \label{eq:harmonic-appx-2}
    \hskeleton,. = \skeleton ,X_{j-1}.\cup\bigcup_{Q\in\gates X_j.}\partial Q;
  \end{equation}
  note that for $l\ge j$ ($l=\infty$ being admissible) $\hskeleton,.$
  embedds isometrically in $X_l$.
  \par Let $f:X_\infty\to\real$ be Lipschitz. We define the
  \textbf{piecewise harmonic approximations} of $f$ as follows. For
  $j\ge0$ let $\gharmonic ,.:X_j\to\real$ be the piecewise harmonic
  function which is harmonic inside each cell of $X_j$ and agrees with
  $f$ on $\skeleton ,.$.
  \par For $j\ge1$ let $\hharmonic ,.:X_j\to\real$ be the piecewise
  harmonic function which agrees with $f$ on $\hskeleton ,X_j.$ and
  such that:
  \begin{enumerate}
  \item For $Q\in\todoub X_{j-1}.$ let $K\greendec$ and $K\reddec$ be
    the lifts of $K_Q$ in $\dym n=n(j),Q.$ and $\gtsheet\greendec$,
    $\gtsheet\reddec$ the corresponding gates; set $Q\greendec = Q\setminus
    K_Q\cup K\greendec\setminus \gtsheet\greendec$ and $Q\reddec = Q\setminus
    K_Q\cup K\reddec\setminus \gtsheet\reddec$. Then $\hharmonic,.$ is
    harmonic in the interior of $Q\greendec\cup Q\reddec$ and agrees
    with $f$ on $\partial Q\cup\partial \gtsheet\greendec\cup\partial
    \gtsheet\reddec$.
  \item For each $Q\in\gates X_j.$ $\hharmonic,.$ is harmonic in the
    interior of $Q$.
  \item Each $Q\in\subdiv X_{j-1}.$ gets isometrically lifted in $X_j$
    and $\hharmonic ,.$ is harmonic in the interior of $Q$ and agrees
    with $f$ on $\partial Q$.
  \end{enumerate}
\end{defn}
  We prove that $\hharmonic,.$ and $\gharmonic,.$ are continuous and
  have distributional derivatives in $L^2$.
  \begin{lem}[Regularity of piecewise harmonic approximations]
    \label{lem:harmonic-reg}
    The function $\hharmonic,.$ and $\gharmonic,.$ are continuous, are
    in $W^{1,2}(X_j,\mu_j)$ and satisfy the energy bounds
      \begin{equation}
    \label{eq:harmonic-appx-3}
    \sup_j\biggl\{
    	\int_{X_j}\|\nabla \hharmonic,.\|_2^2\,d\mu_j,
        \int_{X_j}\|\nabla\gharmonic,.\|_2^2\,d\mu_j
    \biggr\}\le(\glip f.)^2.
  \end{equation}
  \end{lem}
  \begin{proof}
    We are gluing functions which are harmonic in the interia of the
  cells of $X_j$ using compatible boundary conditions. Thus continuity
  and membership in $W^{1,2}(X_j,\mu_j)$ follow if we show that the
  problem:
  \begin{equation}
    \label{eq:harmonic-reg-1}
    \begin{cases}
      \Delta h &= 0\quad\text{in $\Omega$}\\
      h &=G\quad\text{on $\partial \Omega$}
    \end{cases}
  \end{equation}
  where $\Omega$ is either a cube or a cube with an inner smaller cube
  with the same center
  removed (a ``cubular annulus'') and where $G|\partial\Omega$ is
  Lipschitz has a solution which is $C^0(\bar\Omega)$, i.e.~it is
  continuous up to the boundary. We use Perron's Method~\cite[Sec.~2.8]{gilbarg-elliptic-pde}; the desired solution exists as $\Omega$
  has the \emph{exterior cone property:} for each $p\in\partial\Omega$
  there is a small cone:
  \begin{equation}
    \label{eq:harmonic-reg-2}
    C_p = \{p+v: \|v\|_2\le r_p, \text{angle}(v,e_p)\le\alpha_p\}
  \end{equation}
  such that $\ball p,r_p/2.\cap \Omega\subset\ball p,r_p/2.\setminus
  C_p$ and $C_p\cap\partial\Omega=\{p\}$. Then one has to construct~\cite[Ex.~2.12]{gilbarg-elliptic-pde}
  a \emph{local barrier} $w_p$ at $p$:
  \begin{enumerate}
  \item $w_p$ is superharmonic in $\Omega\cap\ball p,s.$ for $s>0$;
    here we will content with $w_p$ harmonic.
  \item $w_p>0$ in $(\bar\Omega\setminus\{p\})\cap\ball p,s.$ and $w_p(p)=0$.
  \end{enumerate}
  We set up a spherical coordinate system $(r,\theta,\phi)$ with
  origin at $p$ and axis $\theta=0$ opposite to $e_p$; the Laplacian
  is given by:
  \begin{equation}
    \label{eq:harmonic-reg-3}
    \Delta = \frac{1}{r^2}\frac{\partial}{\partial r}\bigl(
    r^2\frac{\partial}{\partial r}
    \bigr) +
    \frac{1}{r^2\sin^2\theta}\frac{\partial^2}{\partial^2\phi} +
    \frac{1}{r^2\sin\theta}\frac{\partial}{\partial\theta} \bigl(
    \sin\theta\frac{\partial}{\partial\theta}
    \bigr).
  \end{equation}
  We look for $w_p$ harmonic with ansatz $w_p=r^\lambda f(\theta)$
  where $\lambda>0$; we get the Legendre ODE:
  \begin{equation}
    \label{eq:harmonic-reg-4}
    f''(\theta) + \cot\theta f'(\theta) + \lambda(\lambda+1)f(\theta)=0,
  \end{equation}
  and look for $f(\theta)>0$ when
  $\theta\in[0,\pi-\alpha_p]$. As solution we choose
  $f_\lambda(\theta)=P^0_\lambda(\cos\theta)$ where $P^0_\lambda$ is
  the Legendre function of the first kind
  (see \url{http://dlmf.nist.gov/14}). As $\lambda\searrow0$ 
  $P^0_\lambda(\cos\theta)$ converges to the constant function $1$,
  moreover, $P^0_\lambda(\cos\theta)$ is monotonically decreasing from  $1$ for $\theta=0$
  to $-\infty$ for $\theta=\pi$; depending on $\alpha_p$, $\lambda>0$
  can thus be taken sufficiently small to ensure $f_\lambda(\theta)>0$
  in the desired range of $\theta$.
  \par For further reference we also record the rate of convergence of
  $h(x)$ to $G(p)$ as $x\to p$, given the barrier $w_p$; for
  $\|x-p\|_2\le\delta$, from~\cite[Lem.~2.13]{gilbarg-elliptic-pde}:
  \begin{equation}
    \label{eq:harmonic-reg-barrier-3}
    |h(x)-G(p)|\le \glip G.\|x-p\|_2+\frac{2\|G\|_\infty}{\inf_{\|y-p\|_2>\delta}w_p(y)}w_p(x).
  \end{equation}
  In particular, as we can translate and rotate the same barrier $w_p$
  at the different points of $\partial\Omega$ we get a uniform
  estimate on the convergence of $h$ to $G$.
  \par Finally, harmonic functions minimize
  the $l^2$-energy in the class of functions satisfying their boundary
  conditions. We could then have taken a MacShane extension $\tilde f$
  of $f|\skeleton,.$ or $f|\hskeleton,.$ and
  thus~(\ref{eq:harmonic-appx-3}) follows.
  \end{proof}
\begin{thm}[Rate of collapse of the gates]
  \label{thm:gate-collapse}
  Let $f:X_\infty\to\real$ be Lipschitz. Let $Q\in\todoub X_{l-1}.$
  where $\bar n_{k}+1\le l\le \bar n_{k+1}$ and let $\gtsheet\greendec$,
  $\gtsheet\reddec$ be the two gates of $\dym n_{k+1},Q.$. Having fixed
  $\varepsilon>0$ we say that $Q$ is \textbf{bad} and write
  $Q\in\badset.$ if:
  \begin{equation}
    \label{eq:gate-collapse-s1}
    \biggl|
    	\av_{\partial \gtsheet\greendec}f\,d\hmeas 2. - \av_{\partial
          \gtsheet\reddec}f\,d\hmeas 2.
    \biggr| \ge \frac{\varepsilon}{256\times n_{k+1}}\diam(Q).
  \end{equation}
  Then there is a univeral constant $C$, independent of $f$ and
  $\varepsilon$ such that we have the following estimate on the
  measure of the bad cubes:
  \begin{equation}
    \label{eq:gate-collapse-s2}
    \sum_{k\ge 1}\sum_{l=\bar n_k+1}^{\bar
      n_{k+1}}\sum_{Q\in\badset.}\frac{\varepsilon^2}{n^3_{k+1}}\mu_{l-1}(Q)
    \le C\times(\glip f.)^2.
  \end{equation}
\end{thm}
\begin{proof}
  \MakeStep{Step 1: Orthogonality relations.}
  \par We show that:
  \begin{equation}
    \label{eq:gate-collapse-p1}
    \int_{X_j}(\nabla\hharmonic,.-\nabla\gharmonic,.)\cdot\nabla\hharmonic,.\,d\mu_j=0.
  \end{equation}
  Let $Q\in\cell X_{j-1}.$ and in the first case assume that
  $Q\in\todoub X_{j-1}.$. Let $K_Q$ be the central subcube to be
  doubled in passing from $Q$ to $\dym n=n(j),Q.$ and $K\reddec$,
  $K\greendec$ the two copies of $K_Q$ and $\gtsheet\reddec$, $\gtsheet\greendec$
  the corresponding gates. Write:
  \begin{equation}
    \label{eq:gate-collapse-p2}
    \dym n=n(j),Q. = Q\setminus K_Q \cup K\greendec\setminus
    \gtsheet\greendec \cup K\reddec\setminus \gtsheet\reddec \cup \gtsheet\greendec \cup \gtsheet\reddec;
  \end{equation}
  define $A\greendec = Q\setminus K_Q\cup K\greendec\setminus
  \gtsheet\greendec$ and $A\reddec = Q\setminus K_Q\cup K\reddec\setminus
  \gtsheet\reddec$. As both $\hharmonic,.$ and $\gharmonic,.$ are harmonic in
  the interior of $\gtsheet\greendec \cup \gtsheet\reddec$ and agree on $\partial
  \gtsheet\greendec \cup \partial \gtsheet\reddec$, we have:
  \begin{equation}
    \label{eq:gate-collapse-p3}
    \int_{\gtsheet\greendec \cup
      \gtsheet\reddec}(\nabla\hharmonic,.-\nabla\gharmonic,.)\cdot
    \nabla\hharmonic,.\,d\mu_j=0,
  \end{equation}
  minding that harmonic functions are determined by their boundary
  conditions.
  \par Fix a smooth harmonic function $\phi$ on $A\greendec$ which can
  be extended to be smooth on a neighborhood of $A\greendec$. By
  Lemma~\ref{lem:harmonic-reg} $\gharmonic,.$ and $\hharmonic,.$ are
  in $W^{1,2}(X_j,\mu_j)$ and so we can integrate by parts; as
  $Q\setminus K_Q$ appears both in $A\greendec$ and $A\reddec$ we
  halve $\mu_j$ on it, getting $\tilde \mu_j$ which is 
  just a constant multiple of Lebesgue measure on the whole of $A\greendec$; denoting by
  $\partial_\nu$ the normal derivative on the boundary we get:
  \begin{equation}
    \label{eq:gate-collapse-p4}
    \begin{split}
    \int_{A\greendec}(\nabla\hharmonic,.-\nabla\gharmonic,.)\cdot\nabla\phi\,d\tilde\mu_j &=
    -\int_{A\greendec}(\hharmonic,.-\gharmonic,.)\Delta\phi\,d\tilde\mu_j\\
    &\mskip\munsplit+ \frac{\tilde\mu_j(A\greendec)}{\lebmeas 3.(A\greendec)}\int_{\partial
      A\greendec}(\hharmonic,.-\gharmonic,.) \partial_\nu\phi\,d\hmeas
    2..
  \end{split}
\end{equation}
As $\Delta\phi=0$ in $A\greendec$ and
$\hharmonic,.=\gharmonic,.$ on $\partial A\greendec$ we conclude that:
\begin{equation}
  \label{eq:gate-collapse-p5}
  \int_{A\greendec}(\nabla\hharmonic,.-\nabla\gharmonic,.)\cdot\nabla\phi\,d\tilde\mu_j=0.
\end{equation}
We enlarge the ``cubical annulus'' $A\greendec$ to a slightly larger cubical annulus
$A\egreendec$ which lies in a $(6\varepsilon)$-neighborhood of
$A\greendec$. We then choose $\tilde f_\varepsilon:\partial
A\egreendec\to\real$ to be $1$-Lipschitz on $\partial A\egreendec$ and
such that the graphs of $f|\partial A\greendec$ and $\tilde
f_\varepsilon$ are at Hausdorff distance $\le 150\varepsilon$. Let
$\phi_\varepsilon$ be the harmonic function defined on $A\egreendec$
which agrees with $\tilde f_\varepsilon$ on $\partial A\egreendec$. By
the barrier estimate~(\ref{eq:harmonic-reg-barrier-3}) and the maximum
principle $\phi_\varepsilon\to\hharmonic,.$ uniformly on $A\greendec$
as $\varepsilon\to0$ (here in $A\greendec$ we also include its
boundary). Moreover, as the boundary conditions are $1$-Lipschitz, the
family $\{\nabla\phi_\varepsilon\}_\varepsilon$ is bounded in
$L^2(A\greendec,\tilde\mu_j)$ and we conclude that
$\nabla\phi_\varepsilon$ converges weakly to $\nabla\hharmonic,.$ in
$L^2(A\greendec,\tilde\mu_j)$ as $\varepsilon\to0$. Thus, as each
$\phi_\varepsilon$ satisfies the
orthogonality~(\ref{eq:gate-collapse-p5}), we get:
\begin{equation}
  \label{eq:gate-collapse-p-ins1}
  \int_{A\greendec}(\nabla\hharmonic,.-\nabla\gharmonic,.)\cdot\nabla\hharmonic,.\,d\tilde\mu_j=0,
\end{equation}
and similarly:
\begin{equation}
  \label{eq:gate-collapse-p6}
  \int_{A\reddec}(\nabla\hharmonic,.-\nabla\gharmonic,.)\cdot\nabla\hharmonic,.\,d\tilde\mu_j=0.
\end{equation}
Putting together~(\ref{eq:gate-collapse-p3}),
(\ref{eq:gate-collapse-p5}) and~(\ref{eq:gate-collapse-p6}) we obtain:
\begin{equation}
  \label{eq:gate-collapse-p7}
  \int_{\projinv j,j-1.(Q)}(\nabla\hharmonic,.-\nabla\gharmonic,.)\cdot\nabla\hharmonic,.\,d\mu_j=0.
\end{equation}
The second case is when $Q\in\subdiv X_{j-1}.$; here $Q$ just gets
isometrically lifted in $X_j$ and by construction $\hharmonic,.$ is
harmonic in its interior while $\hharmonic,.$ and $\gharmonic,.$
agree on $\partial Q$. Using again the integration by parts argument
and the smoothing of $\hharmonic,.$ via $\phi_\varepsilon$
we conclude that:
\begin{equation}
  \label{eq:gate-collapse-p8}
  \int_Q(\nabla\hharmonic,.-\nabla\gharmonic,.)\cdot\nabla\hharmonic,.\,d\mu_j=0.
\end{equation}
Combining~(\ref{eq:gate-collapse-p7}) and~(\ref{eq:gate-collapse-p8})
(\ref{eq:gate-collapse-p1}) follows.
\par We now show that:
\begin{equation}
  \label{eq:gate-collapse-p9}
  \int_{X_j}\nabla(\gharmonic j-1,.\circ\proj j,j-1.)\cdot(
  	\nabla\hharmonic,. - \nabla(\gharmonic j-1,.\circ\proj j,j-1.)
  )\,d\mu_j=0.
\end{equation}
Let $Q\in\cell X_{j-1}.$ and consider any lift $\tilde Q$ of $Q$ in
$X_j$ (if $Q\in\todoub X_{j-1}.$ there are two such lifts, a green and
a red one, otherwise there is just one). Now $\hharmonic,.$ and
$\gharmonic j-1,.\circ\proj j,j-1.$ agree on $\partial\tilde
Q=\partial Q$; let $\phi$ be a smooth harmonic function defined on
$\tilde Q$ that can be extended to be smooth on a neighborhood of
$\tilde Q$. By Lemma~\ref{lem:harmonic-reg} $\hharmonic,.$ and
$\gharmonic j-1,.\circ\proj j,j-1.$  are in $W^{1,2}(X_j,\mu_j)$ and
thus we can integrate by parts:
\begin{multline}
  \label{eq:gate-collapse-p10}
    \int_{\tilde Q}\nabla\phi\cdot(
  	\nabla\hharmonic,. - \nabla(\gharmonic j-1,.\circ\proj j,j-1.)
  )\,d\lebmeas 3. \\ = \int_{\partial \tilde Q}(
  	\hharmonic,. - \gharmonic j-1,.\circ\proj j,j-1.)\partial_\nu
        \phi\, d\hmeas 2.\\
         -\int_{\tilde Q}(
  	\hharmonic,. - \gharmonic j-1,.\circ\proj
        j,j-1.)\Delta\phi\, d\lebmeas
    3.=0.
  \end{multline}
Now enlarge $\tilde Q$ to a slightly larger cube $\tilde
Q_\varepsilon$ contained in the $(6\varepsilon)$-neighborhood of
$\tilde Q$. We then choose $\tilde f_\varepsilon:\partial
\tilde Q_\varepsilon\to\real$ to be $1$-Lipschitz on $\partial \tilde Q_\varepsilon$ and
such that the graphs of $f|\partial \tilde Q$ and $\tilde
f_\varepsilon$ are at Hausdorff distance $\le 150\varepsilon$. Let
$\phi_\varepsilon$ be the harmonic function on $\tilde Q_\varepsilon$
which equals $\tilde f_\varepsilon$ on $\partial\tilde Q_\varepsilon$.
By the barrier estimate~(\ref{eq:harmonic-reg-barrier-3}) and the maximum
principle $\phi_\varepsilon\to\gharmonic j-1,.\circ\proj j,j-1.$ uniformly on $\tilde Q$
as $\varepsilon\to0$ (here in $\tilde Q$ we also include its
boundary). Moreover, as the boundary conditions are $1$-Lipschitz, the
family $\{\nabla\phi_\varepsilon\}_\varepsilon$ is bounded in
$L^2(\tilde Q,\lebmeas 3.)$ and we conclude that
$\nabla\phi_\varepsilon$ converges weakly to $\nabla\gharmonic j-1,.\circ\proj j,j-1.$ in
$L^2(\tilde Q,\lebmeas 3.)$ as $\varepsilon\to0$. Thus, as each
$\phi_\varepsilon$ satisfies~(\ref{eq:gate-collapse-p10}), we conclude that:
\begin{equation}
  \label{eq:gate-collapse-p-inv-10}
    \int_{\tilde Q}(\nabla\gharmonic j-1,.\circ\proj j,j-1.)\cdot(
  	\nabla\hharmonic,. - \nabla(\gharmonic j-1,.\circ\proj j,j-1.)
  )\,d\lebmeas 3. = 0.
\end{equation}
 Minding that $\mu_j$ is a constant multiple of Lebesgue measure on
 $\projinv j,j-1.(Q)$ for $Q\in\subdiv X_{j-1}.$ and splitting
 $\mu_{j-1}$ in half on the two lifts of $Q\in\todoub X_{j-1}.$ we get
 that~(\ref{eq:gate-collapse-p9}) follows
 from~(\ref{eq:gate-collapse-p-inv-10}).
 \MakeStep{Step 2: The telescopic series.}
 \par For $\psi:X_j\to\real$ define:
 \begin{equation}
   \label{eq:gate-collapse-p11}
   E[\psi] = \int_{X_j}\|\nabla\psi\|^2\,d\mu_j.
 \end{equation}
 Consider the telescopic series:
 \begin{equation}
   \label{eq:gate-collapse-p12}
   \begin{split}
     \sum_{j=\bar n_1+1}^\infty\bigl\{
     	E[\gharmonic,.]-E[\hharmonic,.]+E[\hharmonic,.]-E[\gharmonic j-1,.]
     \bigr\} &= \lim_{n\to\infty}\bigl(
     	E[\gharmonic n,f.]-E[\gharmonic\bar n_1,.]
     \bigr)\\
     &\le(\glip f.)^2.
   \end{split}
 \end{equation}
 Now
 \begin{equation}
   \label{eq:gate-collapse-p13}
   E[\gharmonic,.] = \int_{X_j}\|\nabla\gharmonic,.-\nabla\hharmonic,.+\nabla\hharmonic,.\|^2\,d\mu_j
 \end{equation}
 and using the orthogonality relation~(\ref{eq:gate-collapse-p1}) we
 get:
 \begin{equation}
   \label{eq:gate-collapse-p14}
   \begin{split}
   E[\gharmonic,.] &= E[\hharmonic,.] +
   \int_{X_j}\|\nabla\gharmonic,.-\nabla\hharmonic,.\|_2^2\,d\mu_j\\
   &\ge E[\hharmonic,.].
 \end{split}
\end{equation}
Similarly, using the orthogonality
relation~(\ref{eq:gate-collapse-p9}) and that $E[\gharmonic j-1,.] =
E[\gharmonic j-1,.\circ\proj j,j-1.]$ we get:
\begin{equation}
  \label{eq:gate-collapse-p15}
  E[\hharmonic,.]-E[\gharmonic j-1,.] =
  \int_{X_j}\|\nabla\hharmonic,.-\nabla\gharmonic j-1,. \circ\proj j,j-1.\|_2^2\,d\mu_j.
\end{equation}
Therefore from~(\ref{eq:gate-collapse-p12}) we have:
\begin{equation}
  \label{eq:gate-collapse-p16}
   \sum_{k\ge 1}\sum_{l=\bar n_k+1}^{\bar n_{k+1}}\sum_{Q\in\cell
     X_{l-1}.}\mu_{l-1}(Q)\av_{\projinv l,l-1.(Q)}\|\nabla\hharmonic
   l,.-\nabla\gharmonic l-1,.\circ\proj l,l-1.\|_2^2\,d\mu_l\le(\glip f.)^2.
 \end{equation}
 \MakeStep{Step 3: Application of Lemma~\ref{lem:energy-lwbd}.}
 \par Assume that $Q\in\badset.$ and as in \texttt{Step 1} write $\dym
 n_{k+1},Q. = A\greendec \cup A\reddec\cup \gtsheet\greendec \cup
 \gtsheet\reddec$. Let $F=\hharmonic,.-\gharmonic j-1,f.\circ\proj j,j-1.$,
 which is locally Lipschitz in the interior of $A\greendec\cup
 A\reddec$ and such that $F=0$ on $\partial Q$. As $Q$ is bad:
 \begin{equation}
   \label{eq:gate-collapse-p17}
   \biggl|
   	\av_{\partial \gtsheet\greendec}f\,d\hmeas 2. - \av_{\partial \gtsheet\reddec}f\,d\hmeas 2.
   \biggr|\ge\frac{\varepsilon}{256\times n_{k+1}}\diam(Q);
 \end{equation}
 however, $\int_{\partial \gtsheet\greendec}\gharmonic j-1,.\circ\proj
 j,j-1. = \int_{\partial \gtsheet\reddec}\gharmonic j-1,.\circ\proj
 j,j-1.$ and $\hharmonic,.=f$ on $\gtsheet\greendec\cup \gtsheet\reddec$, and so at
 least one of the following two must hold
 \begin{align}
   \label{eq:gate-collapse-p18gr}
   \biggl|
   	\av_{\partial \gtsheet\greendec}F\,d\hmeas 2.\biggr|
        &\ge\frac{\varepsilon}{512\times n_{k+1}}\diam(Q)\\
        \label{eq:gate-collapse-p18rd}
\biggl|
   	\av_{\partial \gtsheet\reddec}F\,d\hmeas 2.\biggr|
        &\ge\frac{\varepsilon}{512\times n_{k+1}}\diam(Q).
      \end{align}
      Without loss of generality assume
      that~(\ref{eq:gate-collapse-p18gr}) holds and apply
      Lemma~\ref{lem:energy-lwbd} to $Q\greendec=A\greendec\cup
      \gtsheet\greendec$ with $\eta = \frac{\varepsilon}{512\times
        n_{k+1}}\diam(Q)$ and $s\simeq \diam(Q)/n_{k+1}$ ($\simeq$
      implies a uniform constant). We thus have:
      \begin{equation}
        \label{eq:gate-collapse-p19}
        \int_{A\greendec}\|\nabla F\|^2\,d\lebmeas
        3.\ge\frac{c\varepsilon^2}{n^3_{k+1}}(\diam Q)^3
      \end{equation}
      where $c$ is a universal constant independent of $\varepsilon$,
      $k$ and $l$. As $\mu_l$ is doubling and a constant multiple of
      Lebesgue measure on each cell of $X_l$, we can deflate $c$ to
  get:
  \begin{equation}
    \label{eq:gate-collapse-p20}
    \av_{\projinv l,l-1.(Q)}\|\nabla F\|_2^2\,d\mu_l\ge\frac{c\varepsilon^2}{n^3_{k+1}}.
  \end{equation}
  Plugging~(\ref{eq:gate-collapse-p20}) in~(\ref{eq:gate-collapse-p16})
  finishes the proof.
\end{proof}
\subsection{The proof of differentiability}
\label{subsec:diff-proof}
\begin{thm}[Differentiation of real-valued Lipschitz functions]
  \label{thm:real-diff}
  Let $f:X_\infty\to\real$ be Lipschitz and $\nabla f$ its horizontal
  gradient. Then $f$ is differentiabe $\mu_\infty$-a.e.~with derivative
  $\nabla f$: i.e.~for $\mu_\infty$-a.e.~$p$ one has:
  \begin{equation}
    \label{eq:real-diff-s1}
    \biglip\bigl(
    f-\langle\nabla f(p),\vec x\rangle
\bigr)(p)=0.
  \end{equation}
\end{thm}
\begin{proof}
  \MakeStep{Step 1: Reduction to a constant derivative.}
  Fix $\varepsilon>0$ and write
  $X_\infty=\Omega\cup\bigcup_{t=1}^\infty K_t$ where:
  \begin{enumerate}
  \item $\mu_\infty(\Omega)=0$.
  \item Each $K_t$ is compact with $\mu_\infty(K_t)>0$ and there is a
    $c_t\in\real^3$ such that:
    \begin{equation}
      \label{eq:real-diff-p1}
      \sup_{p\in K_t}\|\nabla f(p)-c_t\|_2\le\varepsilon.
    \end{equation}
  \end{enumerate}
 Fix one index $t$ and drop it from the notation; let $F=
    f-\langle c,\vec x\rangle$, which is $(5\glip
    f.)$-Lipschitz. Assume that we are able to find a universal $C>0$
    independent of $K=K_t$ and $\varepsilon$ such that whenever
    $S\subset K$ is compact with $\mu_\infty(S)>0$ one can find
    $\tilde S\subset S$ with $\mu_\infty(\tilde S)>0$ and
    \begin{equation}
      \label{eq:real-diff-p2}
      \biglip F\le C\varepsilon\quad\text{on $\tilde S$.}
    \end{equation}
Then an exhaustion argument and letting $\varepsilon\searrow0$ will
yield~(\ref{eq:real-diff-s1}).
\MakeStep{Step 2: Avoiding bad gates.}
\par Fix $S\subset K$ with $\mu_\infty(S)>0$, our goal being to
prove~(\ref{eq:real-diff-p2}), which is to be accomplished in
\texttt{Step 6}. Given $Q\in\todoub X_{l-1}.$ let $\gtsheet\greendec(Q)$,
$\gtsheet\reddec(Q)$be the corresponding pair of gates and $G_\varepsilon(Q)$
denote the $\frac{600\slen(\gtsheet\greendec(Q))}{\varepsilon}$-neighborhood
of $\gtsheet\greendec(Q)\cup \gtsheet\reddec(Q)$. By Theorem~\ref{thm:gate-collapse}
we have:
\begin{equation}
  \label{eq:real-diff-p3}
  \sum_{k\ge 1}\sum_{l=\bar n_k+1}^{\bar n_{k+1}}\sum_{Q\in\badset
    X_{l-1}.}\mu_l(\gtsheet\greendec(Q)\cup \gtsheet\reddec(Q))<\infty;
\end{equation}
as the $\mu_l$ are uniformly doubling (Lemma~\ref{lem:ver-doubling})
we also have:
\begin{equation}
  \label{eq:real-diff-p4}
  \sum_{k\ge 1}\sum_{l=\bar n_k+1}^{\bar n_{k+1}}\sum_{Q\in\badset X_{l-1}.}\mu_l(G_\varepsilon(Q))<\infty;
\end{equation}
in particular, we can find a $k_0=k_0(F,S,\varepsilon)$ such that:
\begin{equation}
  \label{eq:real-diff-p5}
  \sum_{k\ge k_0}\sum_{l=\bar n_k+1}^{\bar n_{k+1}}\sum_{Q\in\badset X_{l-1}.}\mu_l(G_\varepsilon(Q))\le\frac{\mu_\infty(S)}{5}.
\end{equation}
Let:
\begin{equation}
  \label{eq:real-diff-p6}
  X\baddec = \bigcup_{k\ge k_0}\bigcup_{l=\bar n_k+1}^{\bar
    n_{k+1}}\bigcup_{Q\in\badset X_{l-1}.}\projinv\infty,l.(G_\varepsilon(Q))
\end{equation}
and note that $\mu_\infty(S\setminus X\baddec)>0$; thus let $\tilde
S\subset S$ consist of those Lebesgue density points of $S\setminus
X\baddec$ which are also approximate continuity points of $\nabla f$
and hence of $\nabla F$. 
\par Pick $p\in\tilde S$ and for $r>0$ let $\fund p,,.$ be a
fundamental configuration at $p$ at scale $r$ and resolution
$\varepsilon$. To each $q\in\fund p,,.$ associate a horizontal path
(possibly with one jump) $\gamma_q$ as in
Lemma~\ref{lem:good-hz-paths}. We say that $q$ is \textbf{bad} if:
\begin{enumerate}
\item $\gamma_q$ is of the form $\jph=(\ph_{-},\jp,\ph_+)$ with
  $\gates\jp.=\gtsheet\greendec(Q)\cup \gtsheet\reddec(Q)$ for $Q\in\badset
  X_{l-1}.$ where $l\in\{\bar n_k+1,\cdots,\bar n_{k+1}\}$.
\item Letting $\jp = (c_1,c_2)$ either $k<k_0$ or
  $d_{X_\infty}(c_1,c_2)>\varepsilon d_{X_\infty}(p,q)$.
\end{enumerate}
We now argue that we can find $r_0=r_0(\varepsilon)>0$ such that if
$r\le r_0$ then there is \emph{no} such bad $q$. For one thing:
\begin{equation}
  \label{eq:real-diff-p7}
  d_{X_\infty}(p,q)\ge d_{X_\infty}(c_1,c_2)\simeq\frac{\slen(X_{l-1})}{n_{k+1}},
\end{equation}
thus choosing $r_0$ sufficiently small we can guarantee $k\ge
k_0$. Secondly,
\begin{equation}
  \label{eq:real-diff-p8}
  d_{X_\infty}(c_1,c_2)>\varepsilon d_{X_\infty}(p,q)
\end{equation}
would imply $p\in G_\varepsilon(Q)$ contradicting the definition of
$\tilde S$. In the following let $r\le r_0$, pick any $q\in\fund p,,.$
and let $\gamma=\gamma_q$.
\MakeStep{Step 3: $\gamma$ is a horizontal path.}
\par Let $\gamma=(\sigma_1,\sigma_2,\cdots,\sigma_s)$ and let
$\dom\sigma_i$ denote the domain of $\sigma_i$; let $\sigma_i(\cpend)$
and $\sigma_i(\cpstart)$ denote the end and the starting point of
$\sigma_i$. 
\par As $p$ is an approximate continuity point of $\nabla F$, we can
find $r_1=r_1(\varepsilon)\le r_0$ such that for each $\sigma_i$ with
$\len(\sigma_i)\ge \varepsilon^3r/400$ there is another horizontal
segment $\tilde\sigma_i$ satisfying:
\begin{enumerate}
\item $\proj\infty,0.\circ\tilde\sigma_i$ and
  $\proj\infty,0.\circ\sigma_i$ are parallel to the same axis.
\item $\tilde\sigma_i$ has the same domain as $\sigma_i$ and:
  \begin{equation}
    \label{eq:real-diff-p9}
    \begin{aligned}
      d_{X_\infty}(\tilde\sigma_i(\cpstart),\sigma_i(\cpstart)) &\le
      3\varepsilon^3r\\
      d_{X_\infty}(\tilde\sigma_i(\cpend),\sigma_i(\cpend)) &\le
      3\varepsilon^3r\\
      \len(\tilde\sigma_i)&=\len(\sigma_i).
    \end{aligned}
  \end{equation}
\item $\int_{\dom\tilde\sigma_i}\|\nabla
  F\|_2\circ\tilde\sigma_i\,d\lebmeas 1.\le 2\varepsilon\len(\tilde\sigma_i)$.
\end{enumerate}
Then:
\begin{equation}
  \label{eq:real-diff-p10}
  \begin{split}
    |F(p) - F(q)| &\le \sum_{i=1}^s\bigl|
    	F(\sigma_i(\cpend)) - F(\sigma_i(\cpstart))
    \bigr| \\ 
    &\le
    \sum_{\sigma_i:\len(\sigma_i)\ge\frac{\varepsilon^3r}{400}}\bigl\{
    \bigl|
    	F(\sigma_i(\cpend)) - F(\tilde\sigma_i(\cpend))
    \bigr| +
    \bigl|
    	F(\tilde\sigma_i(\cpend)) - F(\tilde\sigma_i(\cpstart))
    \bigr| \\    &\mskip\munsplit +
    \bigl|
    	F(\sigma_i(\tilde\cpstart)) - F(\sigma_i(\cpstart))
    \bigr|   
    \bigr\} +
    \sum_{\sigma_i:\len(\sigma_i)<\frac{\varepsilon^3r}{400}} 
    \bigl|
    	F(\sigma_i(\cpstart)) - F(\sigma_i(\cpend))
    \bigr|\\
    &\le (5\glip f.)\times(6\varepsilon^3r)\times\#\bigl\{
    	\sigma_i:\len(\sigma_i)\ge\frac{\varepsilon^3r}{400}
    \bigr\}\\ &\mskip\munsplit +
    \sum_{\sigma_i:\len(\sigma_i)\ge\frac{\varepsilon^3r}{400}}\int_{\dom\tilde\sigma_i}\|\nabla
    F\|_2\circ \tilde\sigma_i\,d\lebmeas 1. + 
    \frac{\varepsilon^3r}{400}\times (5\glip f.)\times \#\bigl\{
    	\sigma_i:\len(\sigma_i)<\frac{\varepsilon^3r}{400}
       \bigr\}\\
      &\le 40\varepsilon^3r\times\glip f. + 2C\varepsilon d_{X_\infty}(p,q),
  \end{split}
\end{equation}
where $C$ is the constant from \textbf{(Gd3)} in
Lemma~\ref{lem:good-hz-paths}; as $d_{X_\infty}(p,q)\ge\varepsilon^2r$
we finally get:
\begin{equation}
  \label{eq:real-diff-p11}
  |F(p)-F(q)| \le (40\times\glip f. + 2C)\varepsilon
  d_{X_\infty}(p,q).
\end{equation}
\MakeStep{Step 4: $\gamma$ has a bad jump.}
\par Let $\gamma=(\ph_{-},\jp,\ph_{+})$ where
$\ph_{-}=(\sigma_1,\cdots,\sigma_{s})$,
$\ph_{+}=(\tau_1,\cdots,\tau_t)$ and $\gates\jp.=\gtsheet\greendec(Q)\cup
\gtsheet\reddec(Q)$ for $Q\in\badset X_{l-1}.$. Let $\jp=(c_1,c_2)$; on
$\ph_{-}$ and $\ph_{+}$ we can estimate as in \texttt{Step 3} while by
\texttt{Step 2} we get:
\begin{equation}
  \label{eq:real-diff-p12}
  d_{X_\infty}(c_1,c_2)\le\varepsilon d_{X_\infty}(p,q)
\end{equation}
as $q$ cannot be bad for $r\le r_0$. Thus:
\begin{equation}
  \label{eq:real-diff-p13}
  |F(p)-F(q)|\le (2C + 45\glip f.)\varepsilon d_{X_\infty}(p,q).
\end{equation}
\MakeStep{Step 5: $\gamma$ has a good jump.}
\par In this case $\jp=(c\greendec,c\reddec)$ and:
\begin{equation}
  \label{eq:real-diff-p14}
  \biggl|
	\av_{\partial \gtsheet\greendec(Q)}F\,d\hmeas 2. - 
        \av_{\partial \gtsheet\reddec(Q)}F\,d\hmeas 2.
  \biggr| \le\varepsilon d_{X_\infty}(c\greendec,c\reddec);
\end{equation}
without loss of generality we may assume that $\ph_{-}$ ends at
$c\greendec$ and $\ph_{+}$ starts at $c\reddec$; moreover, recall that
$d_{X_\infty}(c\greendec, c\reddec)\le d_{X_\infty}(p,q)$. As $F$ is
real-valued and continuous, we can find by the Intermediate Value
Theorem (this is essentially the point were this argument breaks down
for $l^2$-valued maps) $z\greendec\in\partial \gtsheet\greendec(Q)$ and
$z\reddec\in\partial \gtsheet\reddec(Q)$ such that:
\begin{equation}
  \label{eq:real-diff-p15}
  \begin{aligned}
    \av_{\partial \gtsheet\greendec(Q)}F\,d\hmeas 2. &= z\greendec,\\
    \av_{\partial \gtsheet\reddec(Q)}F\,d\hmeas 2. &= z\reddec.
  \end{aligned}
\end{equation}
Now as in \texttt{Step 1} of Lemma~\ref{lem:good-hz-paths} we may find
horizontal paths $\ph\greendec$ and $\ph\reddec$ such that:
\begin{enumerate}
\item $\ph\greendec$ joins $c\greendec$ to $z\greendec$, $\ph\reddec$
  joins $z\reddec$ to $c\reddec$.
\item $\len(\ph\greendec) + \len(\ph\reddec)\le 64
  d_{X_\infty}(c\greendec,c\reddec)$.
\item $\ph\greendec\cup\ph\reddec$ contains at most $6$ horizontal
  segments and, trivially, at most $6$ of them can have length $\le\varepsilon^3r/10$.
\end{enumerate}
We can then estimate:
\begin{equation}
  \label{eq:real-diff-p16}
  \begin{split}
    |F(p)-F(q)| &\le |F(p)-F(c\greendec)| +
    |F(c\greendec)-F(z\greendec)|\\
&\mskip\munsplit+|F(z\greendec)-F(z\reddec)| +
|F(z\reddec)-F(c\reddec)|\\
&\mskip\munsplit+|F(c\reddec)-F(q)|.
  \end{split}
\end{equation}
On $|F(p)-F(c\greendec)|$, $|F(c\greendec)-F(z\greendec)|$,
$|F(z\reddec)-F(c\reddec)|$, $|F(c\reddec)-F(q)|$ we apply the
argument of \texttt{Step 3} to $\ph_{-}$, $\ph\greendec$, $\ph\reddec$
and $\ph_{+}$. For $|F(z\greendec)-F(z\reddec)|$ as the jump is not
bad, i.e.~(\ref{eq:real-diff-p14}):
\begin{equation}
  \label{eq:real-diff-p17}
  |F(z\greendec) - F(z\reddec)|\le\varepsilon d_{X_\infty}(p,q).
\end{equation}
Thus:
\begin{equation}
  \label{eq:real-diff-p18}
  |F(p)-F(q)|\le(800\glip f. + 2C + 8)\varepsilon d_{X_\infty}(p,q).
\end{equation}
\MakeStep{Step 6: The proof of~(\ref{eq:real-diff-p2}).}
\par By \texttt{Steps 3--5} (i.e.~by (\ref{eq:real-diff-p11}),
(\ref{eq:real-diff-p13}) and~(\ref{eq:real-diff-p18})) we have:
\begin{equation}
  \label{eq:real-diff-p19}
  \sup_{q\in\fund p,,.}|F(p)-F(q)|\le (800\glip f. + 2C +
  8)\varepsilon d_{X_\infty}(p,q). 
\end{equation}
Let $\tilde q\in\uball X_\infty,p,r.$ and
  find $q\in\fund p,,.$ with $d_{X_\infty}(q,\tilde q)\le 5\varepsilon
  r$. Then:
  \begin{equation}
    \label{eq:real-diff-p20}
    |F(p)-F(\tilde q)| \le (800\glip f. + 2C+8)\varepsilon
    d_{X_\infty}(p,q) + 25\varepsilon\glip f.r;
  \end{equation}
  thus~(\ref{eq:real-diff-p2}) holds for a universal constant $C$
  independent of $\varepsilon$ and $S$.
\end{proof}
\section{Differentiability of Hilbert-valued Lipschitz maps}
\label{sec:hilb-diff}
\MyOutline{
\begin{enumerate}
\item Modification 1: Energy estimate: take deviation for paired
  points on a small
  corona around the ball where you can get difference between
  comparable balls (might be very small given we lose control on the
  Lipschitz constant)
\item Modification 2: recenter the ball using a biLipschitz movement
  (constant vector + cut-off function)
\end{enumerate}
}
For an $l^2$-valued Lipschitz $f:X_\infty\to l^2$ the argument is more
technical. First recall that $l^2$ has the Radon-Nikodym property: any
Lipschitz $G:\real^n\to l^2$ is differentiabile $\lebmeas n.$-a.e.,
here $n$ being arbitrary. Thus as in Definition~\ref{defn:hz-grad} we
can construct the horizontal gradient of $\nabla f$. As in
Definition~\ref{defn:harmonic-appx} and minding the discussion on
$l^2$-valued harmonic functions in~\ref{defn:harmonic-funs} we can
construct the piecewise harmonic approximations $\hharmonic,.$ and
$\gharmonic,.$. Now the results of Lemma~\ref{lem:harmonic-reg} extend to
this setting. For the boundary regularity we approximate the boundary
conditions with ones that take value in finite-dimensional subspaces
of $l^2$ and use that if $\phi$ is $l^2$-valued and harmonic, then
$\|\phi\|_2^2$ is subharmonic and so we can apply the maximum
principle. For the energy estimates~(\ref{eq:harmonic-appx-3}) we use the
Kirszbraum extension theorem on the cells of $X_j$.
\par However, \texttt{Step 5} in
Theorem~\ref{thm:real-diff} breaks down as we cannot apply the
Intermediate Value Theorem to $f$. To fix it we need to change the
definition of a \emph{bad} cube in Theorem~\ref{thm:gate-collapse}. This
leads to the Theorem~\ref{thm:l2-gate-collapse} which is proved in the
same way (hence the proof is omitted) provided one has a suitable
lower bound on the energy as in Lemma~\ref{lem:l2-energy-lwbd}. The
proof of this Lemma is a bit more technical than the
one~\ref{lem:energy-lwbd}; moreover we get a worse dependence on
$\eta$ and we will have to change the exponent of $\varepsilon$ from
$2$ (in Theorem~\ref{thm:gate-collapse}) to a $4$ (in
Theorem~\ref{thm:l2-gate-collapse}). 
Finally we are able to prove
Theorem~\ref{thm:l2-diff}, where we have just to fix \texttt{Step 5} in
Theorem~\ref{thm:real-diff}.
\begin{lem}[Lower bound on the energy]
  \label{lem:l2-energy-lwbd}
  Let $Q$ be a cube with sidelength $\slen(Q)$ and for
  $s\in(0,\frac{\slen(Q)}{6}]$ let $sQ$ denote the cube with the same
  center as $Q$ and with sidelength $s$. Assume that
  $F:\overline{Q\setminus sQ}\to l^2$ is continuous and locally
  Lispchitz in $Q\setminus sQ$; assume also that the restriction
  $F|\partial(sQ)$ is Lipschitz, that $F=0$ on $\partial Q$ and that there is a $p\in\partial(sQ)$
  such that:
  \begin{equation}
    \label{eq:l2-energy-lwbd-s1}
    \hilbnrm F(p).\ge\eta s.
  \end{equation}
  Then there is a universal constant $c\hardec$ that depends only on
  the Lipschitz constant of $F|\partial (sQ)$ such that
  \begin{equation}
    \label{eq:l2-energy-lwbd-s2}
    \int_{Q\setminus sQ}\hilbnrm \nabla F.^2\,d\lebmeas 3.\ge c\hardec\eta^4 s^3.
  \end{equation}
\end{lem}
\begin{proof}
  As in \texttt{Step 1} of Lemma~\ref{lem:energy-lwbd} we reduce to
  the case of concentric balls where $F:\ball
  0,\frac{\slen(Q)}{2}.\setminus \ball 0,\frac{s}{2}.\to l^2$, that
  $F=0$ for $r=\frac{\slen(Q)}{2}$ and that for some $p\in\partial\ball
  0,\frac{s}{2}.$ one has $\hilbnrm F(p).\ge\eta s$.
  \MakeStep{Step 1: Weighted symmetrization.}
  Now write $p=\frac{s}{2}\omega_0$ where $\omega_0\in\sphere^2$;
  using that $F|\partial\ball 0,\frac{s}{2}.$ is Lipschitz we can find
  a $c>0$ depending only on the Lipschitz constant of $F|\partial\ball
  0,\frac{s}{2}.$ such that  if $\|\omega-\omega_0\|_2\le c\eta$ then:
  \begin{equation}
    \label{eq:l2-energy-lwbd-p1}
    \hilbnrm F\biggl(\frac{s}{2}\omega\biggr).\ge\frac{\eta s}{2}.
  \end{equation}
\par Let $\varphi_\eta$ be a smooth probability distribution on
$\sphere^2$ supported in $\uball \omega_0,c\eta,\sphere^2.$ and such
that:
\begin{equation}
  \label{eq:l2-energy-lwbd-p2}
  \varphi_\eta\le1000\frac{ c^{-2}\eta^{-2} }{\hmeas 2.(\sphere^2)};
\end{equation}
define\begin{equation}
    \label{eq:l2-energy-lwbd-p3}
    \tilde F(r,\omega)=\tilde F(r) =
    \int_{\sphere^2}F(r,\tilde\omega)\varphi_\eta(\tilde\omega)\,d\hmeas 2.(\tilde\omega);
  \end{equation}
  as $F$ is locally Lipschitz in $\ball 0,\frac{\slen(Q)}{2}.\setminus
  \ball 0, \frac{s}{2}.$ so is $\tilde F$. Moreover $\hilbnrm\tilde
  F(s/2).\ge\eta s/2$. We now have:
  \begin{equation}
    \label{eq:l2-energy-lwbd-p4}
    \begin{split}
      \int_{\ball 0,\frac{\slen(Q)}{2}.\setminus\ball 0,
        \frac{s}{2}.}\|\nabla\tilde F(r,\omega)\|_2^2\,r^2dr d\omega
      &= \int_{s/2}^{\slen(Q)/2}r^2\,dr\int_{\sphere^2}d\omega\hilbnrm{
      			\int_{\sphere^2}\partial_rF(r,\tilde\omega)\varphi_\eta(\tilde\omega)\,d\tilde\omega
      		}.^2\\
      &\le\int_{s/2}^{\slen(Q)/2}r^2\,dr\int_{\sphere^2}d\omega\int_{\sphere^2}\hilbnrm{
      	\partial_r F(r,\tilde\omega)
        }.^2\varphi_\eta(\tilde\omega)\,d\tilde\omega\\
      &\le 1000 c^{-2}\eta^{-2}\int_{\ball 0,\frac{\slen(Q)}{2}.\setminus\ball 0,
        \frac{s}{2}.}\hilbnrm{\nabla F(r,\omega)}.^2\,r^2dr d\omega.
    \end{split}
  \end{equation}
The proof is now completed as in \texttt{Step 3} of Lemma~\ref{lem:energy-lwbd}.
\end{proof}
The following Theorem is proven like Theorem~\ref{thm:gate-collapse}
using Lemma~\ref{lem:l2-energy-lwbd} instead of Lemma~\ref{lem:energy-lwbd}.
\begin{thm}[Rate of collapse of the gates]
  \label{thm:l2-gate-collapse}
  Let $f:X_\infty\to l^2$ be Lipschitz and $Q\in\todoub X_{l-1}.$
  where $\bar n_{k}+1\le l\le \bar n_{k+1}$, and let $\gtsheet\greendec$,
  $\gtsheet\reddec$ be the two gates of $\dym n_{k+1},Q.$. Having fixed
  $\varepsilon>0$ we say that $Q$ is \textbf{bad} and write
  $Q\in\badset.$ if for a pair $y\greendec\in\partial\gtsheet\greendec$,
  $y\reddec\in\partial\gtsheet\reddec$ with $\proj
  l,l-1.(y\greendec)=\proj l,l-1.(y\reddec)$ one has:
  \begin{equation}
    \label{eq:l2-gate-collapse-s1}
    \hilbnrm f(y\greendec) - f(y\reddec).
    	  \ge \frac{\varepsilon}{256\times n_{k+1}}\diam(Q).
  \end{equation}
  Then there is a univeral constant $C$, which depends only on the
  Lipschitz constant of $f$ but not 
  $\varepsilon$, such that we have the following estimate on the
  measure of the bad cubes:
  \begin{equation}
    \label{eq:l2-gate-collapse-s2}
    \sum_{k\ge 1}\sum_{l=\bar n_k+1}^{\bar
      n_{k+1}}\sum_{Q\in\badset.}\frac{\varepsilon^4}{n^3_{k+1}}\mu_{l-1}(Q)
    \le C\times(\glip f.)^2.
  \end{equation}
\end{thm}
\begin{thm}[Differentiation of Hilbert-valued Lipschitz functions]
  \label{thm:l2-diff}
  Let $f:X_\infty\to l^2$ be Lipschitz and let $\nabla f$ be its horizontal
  gradient. Then $f$ is differentiable $\mu_\infty$-a.e.~with derivative
  $\nabla f$: i.e.~for $\mu_\infty$-a.e.~$p$ one has:
  \begin{equation}
    \label{eq:l2-diff-s1}
    \biglip\bigl(
    f-\langle\nabla f(p),\vec x\rangle
\bigr)(p)=0.
  \end{equation}
\end{thm}
\begin{proof}
  From the proof of Theorem~\ref{thm:real-diff} we have only to modify
  the argument for \texttt{Step 5} as we cannot use the intermediate
  value theorem. However, as $\jp$ is not bad we know that for each
  pair $z\greendec\in\partial\gtsheet\greendec(Q)$,
  $z\reddec\in\partial\gtsheet\reddec(Q)$ with $\proj
  l,l-1.(z\greendec)=\proj l,l-1.(z\reddec)$ we have:
  \begin{equation}
    \label{eq:l2-diff-p1}
\begin{split}
  \hilbnrm f(z\greendec) -
  f(z\reddec).&\le\frac{\varepsilon}{256\times n_{k+1}}\diam(Q)\\
  &\le\varepsilon d_{X_\infty}(c_1,c_2)\le \varepsilon
  d_{X_\infty}(z\greendec, z\reddec),
\end{split}
\end{equation}
and $d_{X_\infty}(z\greendec,z\reddec)\le
27d_{X_\infty}(c_1,c_2)$. Having chosen such a pair we can then argue
as in \texttt{Step 5} of Theorem~\ref{thm:real-diff}.
\end{proof}
\bibliographystyle{alpha}
\bibliography{tritanopia-biblio}

\end{document}